\documentclass[reqno,a4paper,12pt]{amsbook}
\usepackage{amsmath,amsthm,amsfonts,amssymb}
\usepackage{mathrsfs,dsfont}

\renewcommand\labelenumi{\theenumi}

\allowdisplaybreaks 

\def\C{\mathbb{C}}
\def\G{\mathcal{G}}
\def\K{\mathcal{K}}

\def\R{\mathbb{R}}
\def\S{\mathcal{S}}
\def\T{\mathbb{T}}
\def\X{\mathscr{X}}
\def\Z{\mathbb{Z}}
\def\1{\mathds{1}}
\def\eps{\varepsilon}
\def\U{\mathscr{U}}

\newcommand\ft[1]{\widehat{#1}}
\newcommand\dotprod[2]{\langle #1 , #2 \rangle}
\newcommand\var{\operatorname{\mathscr{V}}}

\newcommand\supp{\operatorname{supp}}
\newcommand\prob{\operatorname{\mathscr{P}}}
\newcommand\E{\operatorname{\mathscr{E}}}

\newcommand\lip{\operatorname{Lip}}

\newcommand\deltatwo{\mathit{\Delta}_2}
\newcommand\lowers{\frac1{4}}
\newcommand\uppers{\frac1{3}}

\renewcommand\leq{\leqslant}
\renewcommand\geq{\geqslant}

\theoremstyle{plain}
\newtheorem{theorem}{Theorem}[chapter]
\newtheorem*{theorem*}{Theorem}
\newtheorem{lemma}{Lemma}[chapter]

\newtheorem{proposition}[lemma]{Proposition}
\newtheorem{corollary}[lemma]{Corollary}

\theoremstyle{remark}
\newtheorem*{remark}{Remark}

\newtheorem*{remarks*}{Remarks}

\newenvironment{remarks}
{\begin{remarks*}
\begin{enumerate}
\addtolength{\itemsep}{2pt}
\renewcommand\theenumi{\arabic{enumi}.}}
{\end{enumerate}
\end{remarks*}}

\newenvironment{enumerate-math}
{\begin{enumerate}
\addtolength{\itemsep}{5pt}
\renewcommand\theenumi{(\roman{enumi})}}
{\end{enumerate}}

\newenvironment{enumerate-math-abc}
{\begin{enumerate}
\addtolength{\itemsep}{5pt}
\renewcommand\theenumi{(\alph{enumi})}}
{\end{enumerate}}

\makeatletter
\newcommand\@dotsep{4.5}
\def\@tocline#1#2#3#4#5#6#7{\relax
  \ifnum #1>\c@tocdepth 
  \else
    \par \addpenalty\@secpenalty\addvspace{#2}%
    \begingroup \hyphenpenalty\@M
    \@ifempty{#4}{%
      \@tempdima\csname r@tocindent\number#1\endcsname\relax
    }{%
      \@tempdima#4\relax
    }%
    \parindent\z@ \leftskip#3\relax \advance\leftskip\@tempdima\relax
    \rightskip\@pnumwidth plus1em \parfillskip-\@pnumwidth
    #5\leavevmode\hskip-\@tempdima #6\relax
    \leaders\hbox{$\m@th
      \mkern \@dotsep mu\hbox{.}\mkern \@dotsep mu$}\hfill
    \hbox to\@pnumwidth{\@tocpagenum{#7}}\par
    \nobreak
    \endgroup
  \fi}
\makeatother

\begin{document}

\frontmatter

\begin{titlepage}
\begin{center}
\begin{large}
Tel-Aviv University
\par\vspace{0.25cm}
School of Mathematical Sciences
\par\vspace{2.5cm}
\textbf{\LARGE Piatetski-Shapiro's phenomenon}
\par\vspace{0.25cm}
\textbf{\LARGE and related problems}
\par\vspace{2cm}
Thesis submitted for the degree ``Doctor of Philosophy''
\par\vspace{0.25cm}
by
\par\vspace{0.25cm}
\textbf{Nir Lev}
\par\vspace{2.5cm}
Prepared under the supervision of
\par\vspace{0.25cm}
Prof. Alexander Olevskii
\par\vspace{2cm}
Submitted to the senate of Tel-Aviv University
\par\vspace{0.25cm}
August 2008
\end{large}
\end{center}
\end{titlepage}

\mainmatter

\setcounter{page}{2}

\thispagestyle{empty}
\vspace*{2.5cm}
\begin{center}
\begin{large}
This Ph.D. thesis was prepared under the supervision of
\par\vspace{0.25cm}
\textbf{Professor Alexander Olevskii}
\end{large}
\end{center}

\chapter*{Acknowledgments}

I wish to express my sincere gratitude to my supervisor Professor Alexander
Olevskii, who introduced me into mathematical research, and who has been for me
a role model of a mathematician and a teacher. I am indebted to him for his
guidance and support, and consider it a privilege to having been his student.

I am deeply grateful to Professor Mikhail Sodin for his kind help and constant
support, and for providing me with invaluable advice.

I also wish to thank my friend Ron Peled, for his persistent willingness to encourage
and help, and for stimulating mathematical discussions.


\tableofcontents


\chapter*{Introduction}

This thesis is concerned with some problems in two areas of Fourier Analysis:
uniqueness theory of trigonometric expansions, and the theory of translation
invariant subspaces in function spaces.

Our main result in the first area extends to $\ell_q$ spaces ($q > 2$) a deep
phenomenon found by Piatetski-Shapiro in 1954 for the space $c_0$.

The approach we developed also enabled us to get a result in the second mentioned
area, which a priori does not look connected with the first one. The result (maybe,
a bit surprising) is: one cannot characterize the functions in $\ell_p(\Z)$ or
$L^p(\R)$, $1 < p < 2$, whose translates span the whole space, by the zero set of
their Fourier transform. This should be contrasted against the classical Wiener
theorems related to the cases $p=1,2$.

Below we give a detailed exposition of the main results proved in this thesis.


\section{Piatetski-Shapiro's Phenomenon}

\subsection*{Background}

The theory of uniqueness deals with the question: if a function admits a representation by
a trigonometric series, $f(t) = \sum_{-\infty}^{+\infty} c_n \, e^{int}$, is this representation
unique? The answer depends on the sense in which the convergence of the series is understood.
It was proved by Cantor that if the series converges at every point then the coefficients
$\{c_n\}$ are uniquely defined. On the other hand, for convergence \emph{almost everywhere}
the uniqueness result fails. This was proved by Menshov who constructed a trigonometric
series which converges to zero almost everywhere, but which is not identically zero.

A central concept in this subject is the notion of a \emph{set of uniqueness}. A set $E$ of
measure zero is called a set of uniqueness, if the only trigonometric series which converges
to zero everywhere outside $E$ is the series which is identically zero. Thus if a function
admits a representation by a trigonometric series which converges everywhere outside a set of
uniqueness, then the coefficients of the series are uniquely defined. A set $E$ which is not
a set of uniqueness, is called a \emph{set of multiplicity}. The problem to determine which sets
$E$ of measure zero are sets of uniqueness and which are sets of multiplicity, is very difficult.
It is well-known that a crucial role is played by the arithmetic properties of the set
\cite{bary, zygmund}.

If the set $E$ is \emph{closed} then the following criterion is available: $E$ is a set of
multiplicity if and only if there is a non-zero Schwartz distribution on the circle, which is
supported by $E$, and has Fourier coefficients tending to zero \cite{kahane-salem}. Unfortunately,
this criterion is still difficult to check. However, it is common to use this criterion in order
to show that a given set $E$ is a set of multiplicity, by constructing a 
\emph{measure} which is supported by $E$ and which has Fourier coefficients tending
to zero. The existence of such a measure is a special property of a set of multiplicity, and
in such a case $E$ is called a \emph{set of restricted multiplicity}.

In fact, most known sets of multiplicity are sets of restricted multiplicity,
and it was long believed that the two notions actually coincide. However, in 1954
Piatetski-Shapiro \cite{piatetskii-2} refuted this hypothesis by constructing an example of
a closed set of multiplicity $E$ such that no measure supported by $E$ can have Fourier
coefficients tending to zero. This subject was further developed by K\"{o}rner
\cite{korner:pseudofunc} and Kaufman \cite{kaufman:msets}. During the 1980's, a strong
structural difference between the two notions was found using methods of descriptive
set theory \cite{kechris-louveau}.

\subsection*{Results}
We investigate the following problem. Suppose that $E$ is a closed set on the circle,
which supports a Schwartz distribution $S \sim \sum c_n e^{int}$ whose Fourier coefficients
$\{c_n\}$ belong to a certain space $\X$. Must then $E$ also support a \emph{measure} $\mu$
whose Fourier coefficients belong to the same space $\X$\,?
We say that \emph{Piatetski-Shapiro's phenomenon exists in the space $\X$} if the answer
is negative; that is, if there is a closed set $E$ which supports a distribution with Fourier
coefficients in $\X$, but which does not support such a measure. The theorem of Piatetski-Shapiro
thus states that Piatetski-Shapiro's phenomenon exists in the space $c_0$ of sequences tending
to zero. On the other hand, it is known from potential theory that this is not the case in certain
weighted $\ell_2$ spaces. Precisely, if a closed set $E$ supports a distribution
$S \sim \sum c_n e^{int}$ such that $\sum |n|^{- \alpha} |c_n|^2 < \infty$, $0 < \alpha \leq 1$,
then it also supports a positive measure satisfying this property \cite{kahane-salem}.

What can be said about $\ell_q$ spaces? Only the case $q > 2$ is non-trivial, since only
in this case there exist distributions with coefficients in $\ell_q$ which are not measures.
In our joint work with A. Olevskii \cite{lev-olevskii:piatetski} we proved that Piatetski-Shapiro's
phenomenon does exist in these spaces. That is,

\begin{theorem*}
For any $q > 2$ there is a closed set $E$ such that:
\begin{enumerate-math}
\item
There is a distribution $S \sim \sum c_n e^{int}$ supported by $E$, with
$\sum |c_n|^q < \infty$.
\item
No measure $\mu$ may satisfy this property.
\end{enumerate-math}
\end{theorem*}

As pointed out in \cite{lev-olevskii:piatetski}, the property of our set $E$ can also be formulated
in the language of the uniqueness problem: there is a non-zero trigonometric series $\sum c_n e^{int}$
with $\sum |c_n|^q < \infty$, which converges to zero everywhere outside $E$, but no Fourier-Stieltjes
series may satisfy this property. Our proof is inspired by Kahane's presentation of the K\"{o}rner-Kaufman
results, see \cite[pp. 213--216]{kahane-salem}. The main new ingredients in \cite{lev-olevskii:piatetski}
are Riesz products, and probabilistic concentration estimates.

We continue to study further aspects of Piatetski-Shapiro's phenomenon. For example, how small
can be the set $E$ in the above theorem. In particular, how small can be the Hausdorff dimension
of $E$. Using the ``shrinking method'' of Kaufman we prove the following stronger result: any closed
set which supports a distribution with Fourier coefficients in $\ell_q$,
contains a closed subset which also supports such a distribution, but not such a measure. In particular,
this result enables us to compute the minimal Hausdorff dimension of $E$.
Another question under consideration is whether Piatetski-Shapiro's phenomenon is ``typical'' or ``rare''.
Inspired by K\"{o}rner's ideas \cite{korner:curve} we define the space of all pairs $(E, S)$ where $E$ is
a closed set on the circle, and $S$ is a distribution supported by $E$ with Fourier coefficients in $\ell_q$.
We prove that ``almost all'' pairs $(E,S)$, in the sense of Baire categories, satisfy the above theorem.

Are there spaces other than $c_0$ and $\ell_q$ where Piatetski-Shapiro's phenomenon could exist?
We find a certain class of Orlicz spaces, where this is indeed the case. For example, for any $q > 2$
and $\alpha > 0$, Piatetski-Shapiro's phenomenon exists in the space of sequences $\{c_n\}$ satisfying
$\sum |c_n|^q \log^{\alpha} (e + |c_n|^{-1}) < \infty$.


\section{Generators in $\ell_p$ and Zero Set of Fourier Transform}

\subsection*{Background}

A function $F: \Z \to C$ is called a generator (or a cyclic vector) in the space $\ell_p(\Z)$
if the linear span of its translates is dense in $\ell_p$. In other words, an element 
$F \in \ell_p(\Z)$ is a generator if the closed translation invariant linear subspace generated
by $F$ is the whole $\ell_p$. How to know whether a given $F$ is a generator, or not? The
classical cases are $p = 1,2$. In these cases, Wiener \cite{wiener} characterized the generators
in terms of the zero set $Z_f$ of the Fourier transform
\[
f(t) = \sum_{n \in \Z} F(n) e^{int},
\]
as follows:
\begin{quote}
\emph{$F$ is a generator in $\ell_1$ if and only if $f(t)$ has no zeros.\\[4pt]
$F$ is a generator in $\ell_2$ if and only if $f(t) \neq 0$ almost everywhere.}
\end{quote}
The same characterization is true for $L^1(\R)$ and $L^2(\R)$.

The case $1 < p < 2$ is much less understood. ``Interpolating'' between $p=1$ and $2$, one may
expect that the generators in $\ell_p$ $(1 < p < 2)$ could be characterized by the condition
that the zero set of the Fourier transform is ``small'' in a certain sense. In this context
various metrical, arithmetical and other properties of the zero set for generators and
non-generators have been studied by Beurling \cite{beurling}, Pollard \cite{pollard}, Herz
\cite{herz}, Newman \cite{newman} and other authors. However, none of their results provides a
complete characterization of the generators.

For example, Beurling proved that if the zero set $Z_f$ is sufficiently small in the sense of
Hausdorff dimension, then $F$ is a generator in $\ell_p$. However, this is only under the 
assumption that $F \in \ell_1$. Also, the converse is not true: one can construct a function
$F$ which is a generator in every $\ell_p$ $(1 < p < 2)$, but $Z_f$ has Hausdorff dimension $1$.

\subsection*{Results}
We study the following question: \emph{is it possible at all} to describe the generators
$F$ in $\ell_p$ $(1 < p < 2)$ by the zero set of the Fourier transform? In our joint work with
A. Olevskii \cite{lev-olevskii:generators} we show that, in contrast to the classical cases
$p=1$ and $2$, \emph{the characterization of the generators in $\ell_p$ $(1 < p < 2)$ by the
zero set of their Fourier transforms is impossible in principle.}

Our main result is the following theorem.

\begin{theorem*}
Given $1 < p < 2$ one can find two continuous functions $f$ and $g$
on the circle, with the following properties:
\begin{enumerate-math}
\item They have the same set of zeros.
\item The Fourier transforms $F = \ft{f}$ and $G = \ft{g}$ both belong to $\ell_p(\Z)$,
\item $G$ is a generator in $\ell_p$, but $F$ is not.
\end{enumerate-math}
\end{theorem*}
A similar result is true for $L^p(\R)$ spaces.

We point out the role of the continuity condition in the theorem, which makes precise the
concept of the ``zero set''. Generally, there is no canonical way to define the zero set
of the Fourier transform of an arbitrary element in $\ell_p$ $(1 < p < 2)$.

The proof of the above theorem is based on a modifications and development of the approach
which allowed us to obtain Piatetski-Shapiro's phenomenon in $\ell_q$ spaces $(q > 2)$. In
fact, the result is derived from the construction of a closed set $E$ on the circle, which
satisfies a certain strong form of Piatetski-Shapiro's phenomenon in $\ell_q$, where
$q$ is the exponent conjugate to $p$ satisfying $1/p + 1/q = 1$.

More specifically, let $A_r(\T)$ $(1 \leq r < \infty)$ denote the Banach space of functions or
distributions on the circle with Fourier coefficients in $\ell_r(\Z)$, endowed with the norm
$\|f\|_{A_r} := \|\ft{f}\|_{\ell_r}$. We prove the following theorem.

\begin{theorem*}
Given any $1 < p < 2$ one can construct a closed set $E$ on the circle,
and a function $g \in C(\T) \cap A_p(\T)$, such that:
\begin{enumerate-math}
\item
$Z_g := \{t : g(t)=0\} = E$,
\item
The set $\{P(t)g(t)\}$, where $P$ goes through all trigonometric polynomials,
is dense in $A_p$,
\item
There is a (non-zero) distribution $S$, supported by $E$, which
belongs to $A_q$, where $q = p/(p - 1)$ is the exponent conjugate to $p$.
\end{enumerate-math}
\end{theorem*}

The property (ii) means that $\ft{g}$ is a generator. On the other hand, (iii) is equivalent
to the fact that no Fourier transform of a smooth function $f$ vanishing on $E$, could be a
generator. The first theorem is therefore a consequence of the second one.

We remark that the set $E$ satisfies Piatetski-Shapiro's phenomenon in $\ell_q$, since the
property (ii) implies that $E$ cannot support a non-zero measure $\mu$ whose Fourier coefficients
belong to $\ell_q$.


\numberwithin{section}{chapter}
\numberwithin{equation}{section}


\chapter{Piatetski-Shapiro's Theorem}
\label{chap:c0}

In this chapter we present a simplified proof of Piatetski-Shapiro's theorem.

\section{Introduction}

\subsection{}
\label{subsec:about-piatetski-result}

A compact $K$ on the circle is called a \emph{set of uniqueness} if the only
trigonometric series
\[
\sum_{n \in \Z} c(n) e^{int}
\]
which converges to zero everywhere outside $K$, is the series which is identically
zero. Otherwise, $K$ is called a \emph{set of multiplicity}. It is well-known
that $K$ is a set of multiplicity if and only if there is a non-zero Schwartz
distribution $S$, supported by $K$, with Fourier coefficients $\{\ft{S}(n)\}$
tending to zero (see \cite{kahane-salem}, Chapter V). If there is a non-zero
\emph{measure} $\mu$ satisfying these properties, then $K$ is called a \emph{set
of restricted multiplicity}.

It was conjectured that any set of multiplicity is in fact a set of restricted
multiplicity. For example, a symmetric Cantor set with constant dissection ratio
$\xi$, where $0 < \xi < \tfrac1{2}$, is either a set of uniqueness or a set of restricted
multiplicity (see \cite{kahane-salem}, Chapter VI). However, Piatetski-Shapiro
showed in \cite{piatetskii-2} that in general this is not true. He proved the
following theorem.

\begin{theorem}[Piatetski-Shapiro, 1954]
\label{thm:piatetski-theorem-c0}
There exists a compact $K$ on the circle, which supports a distribution
$S$ with Fourier coefficients tending to zero, but which does not support any
measure with this property.
\end{theorem}

It is interesting to remark that Piatetski-Shapiro proved the above theorem by
constructing the following concrete example. Let $K_\gamma$ $(0 < \gamma < \tfrac1{2})$
denote the set of real numbers $t$ in the segment $[0, 2\pi)$, written in binary expansion
\[
t = 2 \pi \sum_{k=1}^{\infty} \eps_k 2^{-k}, \quad \eps_k \in \{0,1\},
\]
which satisfy
\[
\frac1{n} \sum_{k=1}^{n} \eps_k \leq \gamma, \quad n=1,2,3,\cdots \,.
\]
Then Theorem \ref{thm:piatetski-theorem-c0} holds with $K = K_\gamma$,
see \cite{piatetskii-2} or \cite[Section 4.4]{graham-mcgehee}.

\subsection{}
\label{subsec:extensions-of-piatetski-result}
Piatetski-Shapiro's theorem was further extended in the years 1971--73, in the
following direction. A compact $K$ is called a \emph{Helson set} if any function
defined and continuous on $K$ can be extended to a function on the circle with an
absolutely convergent Fourier series. It was proved by Helson \cite{helson} that
a Helson set can never be a set of restricted multiplicity. The problem whether it
can be a set of multiplicity was solved in the affirmative by K\"{o}rner in
\cite{korner:pseudofunc}. A simpler proof of K\"{o}rner's theorem was later found
by Kaufman \cite{kaufman:msets}, who also proved a stronger result: any set of
multiplicity contains a Helson set of multiplicity.

A compact which is not a set of restricted multiplicity is called \emph{a set
of extended uniqueness}. During the 1980's, descriptive set theory was
used to establish a structural difference between the class $\U$ of the closed
sets of uniqueness, and the class $\U_0$ of the closed set of extended uniqueness.
Namely, Debs and Saint-Raymond \cite{debs-saint-raymond} proved that $\U$ has no
Borel basis, while Kechris and Louveau \cite{kechris-louveau} proved that $\U_0$
does admit such a basis. For a detailed discussion see \cite{kechris-louveau}.

\subsection{}
In this chapter, our goal is to present a simplified proof of Piatetski-Shapiro's
theorem. Our approach is based on Kahane's presentation \cite[pp. 213--216]{kahane-salem}
of the K\"{o}rner-Kaufman results mentioned above. The main new ingredient in the proof
presented below is the use of Riesz products, which allows us to simplify certain
arguments from \cite{kahane-salem} related to Fourier series in multiple variables.


\section{Lemmas}

\subsection{Extended uniqueness}
We start with the formulation of a condition on a compact $K$ which implies
that it supports no measure with Fourier coefficients tending to zero. In fact,
one can relax this condition (see \cite{piatetskii-1}, Theorem 5), so the
following lemma is not the strongest statement which is possible. However, it
will be sufficient for us, and its proof is very simple.

\begin{lemma}
\label{lemma:nomeasure-c0}
Suppose a compact $K$ on the circle satisfies the following condition:
for any $\eps > 0$ and any integer $\nu$ there is a real trigonometric polynomial
\begin{equation}
\label{eq:nomeasure-c0-x}
X(t) = \sum_{|n| \geq \nu} \ft{X}(n) e^{int},
\quad \sum |\ft{X}(n)| \leq 2, \quad
\text{$|1 - X(t)| < \ \eps$ on $K$},
\end{equation}
Then $K$ does not support any non-zero measure $\mu$ such that $\ft{\mu}(n)
\to 0$ as $|n| \to \infty$.
\end{lemma}

%

%

\begin{proof}
Let $\mu$ be a non-zero measure supported by $K$, with Fourier coefficients
tending to zero. Clearly we may assume 
$\int |d \mu| = 1$,
and by mutliplication of $\mu$ by an appropriate exponential $e^{int}$
we may also assume $\ft{\mu}(0) \neq 0$. Given $\eps > 0$, choose $\nu$ such that
$\sup_{|n| \geq \nu} |\ft{\mu}(n)| < \eps / 2$, and let $X$ be a real trigonometric
polynomial satifying \eqref{eq:nomeasure-c0-x}. Then one has
\[
\Big| \int_{\T} X \, d\mu \Big| = \Big| \sum_{|n| \geq \nu}
\ft{X}(n) \, \ft{\mu}(-n) \Big| \leq 2 \sup_{|n| \geq \nu}
|\ft{\mu}(n)| < \eps.
\]
On the other hand $\mu$ is supported by $K$, hence
\[
\Big| \int_{\T} X \, d\mu \Big| = \Big| \int_{K} X \, d\mu \Big|
\geq  \Big| \int_{K} d\mu \Big| - \int_{K} |1 - X| \, d\mu
\geq |\ft{\mu}(0)| - \eps.
\]
Since $\eps$ was arbitrary, it follows that $\ft{\mu}(0) = 0$, a contradiction.
\end{proof}

\subsection{Kahane's lemma}
The following lemma is taken from \cite[p. 214]{kahane-salem}. For completeness,
we include its proof here.

\begin{lemma}[Kahane]
\label{lemma:kahane-lemma}
Let $I$ be a closed interval, $I \subset (0,1)$. Given any $\delta > 0$
there is a (signed) measure $\rho$, supported by a finite subset of $I$,
such that
\begin{equation}
\label{eq:kahane-lemma}
\int d\rho = 1 \qquad \text{and} \qquad
\Big| \int s^k \, d\rho(s) \Big| < \delta
\quad (k=1,2,\dots).
\end{equation}
\end{lemma}

\begin{remark}
Note that the properties \eqref{eq:kahane-lemma} are clearly satisfied
by the Dirac measure at the point $0$; however, the idea of Lemma
\ref{lemma:kahane-lemma} is to obtain \eqref{eq:kahane-lemma} under the
constraint that the support of $\rho$ lies in $I$. Also note that the
measure $\rho$ cannot be chosen positive for arbitrarily small values of
$\delta$, as one can prove using a standard application of Helly's theorem.
\end{remark}


\begin{proof}[Proof of Lemma \ref{lemma:kahane-lemma}]
In the space $c_0$ consider the collection $\S$ of all sequences $(1,s,s^2,s^3,\dots)$
corresponding to $s \in I$ (these sequences belong to $c_0$, since $0 < s < 1$). Let
us show that $\S$ forms a complete system in $c_0$. Indeed, a linear functional over $c_0$
which annihilates all elements of $\S$ is a sequence $\{a_k\} \in \ell_1$ such that
$\sum_{k=0}^{\infty} a_k s^k$ vanishes on $I$, hence $\{a_k\} = 0$. The completeness
therefore follows from the Hahn-Banach theorem. In particular, it is possible to
approximate the sequence $(1,0,0,\dots)$, in the $c_0$ norm, by linear combinations
of elements of $\S$. That is, given $\eta > 0$ there exist $s_1, \dots, s_n \in I$
and $c_1, \dots, c_n \in \R$ such that
\[
\Big|\sum_{j=1}^{n} c_j - 1 \Big| < \eta
\qquad \text{and} \qquad
\Big|\sum_{j=1}^{n} c_j s_j^k \Big| < \eta
\quad (k=1,2,\dots).
\]
If $\eta$ is chosen small enough so that $\eta / (1-\eta) < \delta$, the conditions
\eqref{eq:kahane-lemma} are satisfied by the measure 
\[
\rho = \frac{\sum_{j=1}^{n} c_j \, \delta_{s_j}}{\sum_{j=1}^{n} c_j} \, ,
\]
where $\delta_s$ denotes the Dirac measure at the point $s$.
\end{proof}


\section{Riesz products}

The Riesz product, first introduced by Riesz in 1918, is an important tool in
Fourier analysis (see \cite{zygmund}, Volume I, Chapter V, \textsection 7).
In this section we use finite Riesz products to simplify the construction
from \cite{kaufman:msets}.

\subsection{Notation}
Define a finite Riesz product
\begin{equation}
\label{eq:riesz-prod-c0}
\lambda_s(t) = \prod_{j=1}^{N} \big( 1 + 2s \cos \nu^j t \big)
\end{equation}
where the parameters satisfy $0 < s < \tfrac1{2}$ and $\nu$ is an integer
$\geq 3$. One can write $2 \cos \nu^j t = e^{i \nu^j t} + e^{-i \nu^j t}$,
and then expand the product \eqref{eq:riesz-prod-c0}. This would yield
\begin{equation}
\label{eq:lambda-trig-expansion-c0}
\lambda_s(t) = 1 + \sum_{\tau \neq 0} s^{\sum
|\tau_j|} \, e^{i (\tau_1 \nu + \tau_2 \nu^2 + \dots + \tau_N \nu^N) t},
\end{equation}
where the sum is over all non-zero vectors
\[
\tau = (\tau_1, \dots, \tau_N), \quad \tau_j \in \{-1,0,1\}.
\]
Since $\nu \geq 3$, every integer $n$ admits at most one representation
of the form
\[
n = \tau_1 \nu + \tau_2 \nu^2 + \dots + \tau_N \nu^N,
\]
therefore the frequencies in the sum \eqref{eq:lambda-trig-expansion-c0}
are distinct. So \eqref{eq:lambda-trig-expansion-c0} is the Fourier
expansion of $\lambda_s$.


\subsection{Concentration}

The trigonometric polynomial $\lambda_s$ is non-negative and has integral
$= 1$. We therefore may consider a probability measure $\lambda_s(t)
dt / 2\pi$ on the circle, which we shall also denote (with no risk
of ambiguity) by $\lambda_s$. Define a trigonometric polynomial
\begin{equation}
\label{eq:x-cosines-c0}
X(t) = \frac{2}{N} \sum_{j=1}^{N} \cos \nu^j t.
\end{equation}
We regard $X$ as a random variable with respect to the probability
measure $\lambda_s$.


\begin{lemma}
\label{lemma:lambda-concentration-c0}
Given $\delta > 0$ we have
\[
\lambda_s \{t \in \T : |X(t) - 1| > 3 \delta\} \longrightarrow 0
\]
as $N \to \infty$, uniformly in
$s \in \Big( \frac1{2} - \delta, \frac1{2} \Big)$.
\end{lemma}

\begin{proof}
Denote $X_j(t) = \cos \nu^j t$, then using
\eqref{eq:lambda-trig-expansion-c0} one can compute
\[
\E(X_j) = s, \quad \var(X_j) = \E (X_j - s)^2 = \tfrac1{2} - s^2,
\]
\[
\E (X_j-s)(X_k - s) = 0, \quad j \neq k.
\]
We therefore have
\[
\E(X) = \frac{2}{N} \sum_{j=1}^{N} \E(X_j) = 2s,
\]
and the fact that the $X_j$'s are uncorrelated allows us to compute
\[
\var(X) = \E(X - 2s)^2 =
\frac{4}{N^2} \sum_{j=1}^{N} \var(X_j) = \frac{4}{N} \, \Big(\tfrac1{2} -
s^2\Big) < \frac{2}{N}.
\]
If $s \in \Big( \frac1{2} - \delta, \frac1{2} \Big)$, then
by Chebyshev's inequality
\[
\lambda_s \{ |X(t) - 1| > 3 \delta\} \leq \lambda_s \{|X(t) - 2s| > \delta\}
\leq \delta^{-2} \var(X) \to 0
\]
as $N \to \infty$, uniformly in $s$.
%
\end{proof}


\section{Proof of Piatetski-Shapiro's theorem}

\subsection{}
We can now prove

\begin{lemma}
\label{lemma:principal-c0}
Let $\eps > 0$ and an integer $\nu \geq 3$ be given. Then there exist
a $C^\infty$ function $f: \T \to \C$ and an integer $N$ such that
\begin{enumerate-math}
\item
$f$ is supported by the compact
\begin{equation}
\label{eq:compact-k-def-co}
K = \{t \in \T : |X(t) - 1| \leq \eps\},
\end{equation}
where $X$ is the trigonometric polynomial defined by \eqref{eq:x-cosines-c0}.
\item
$f(t) = 1 + \sum_{n \neq 0} \ft{f}(n)e^{int}$, where
$|\ft{f}(n)|<\eps$ for all $n \neq 0$.
\end{enumerate-math}
\end{lemma}

\begin{proof}
The proof follows a similar line as in \cite[p. 213]{kahane-salem}.
Given $\delta > 0$, choose a measure $\rho$ supported by the interval
$\Big( \frac1{2} - \delta, \frac1{2} \Big)$
according to Lemma \ref{lemma:kahane-lemma}. Define
\[
\lambda(t) = \int \lambda_s(t) \, d\rho(s).
\]
By \eqref{eq:lambda-trig-expansion-c0} and \eqref{eq:kahane-lemma},
$\lambda$ is a trigonometric polynomial with Fourier expansion
\[
\lambda(t) = 1 + \sum_{\tau \neq 0}
\Big\{ \int s^{\sum |\tau_j|} d\rho \Big\} \; e^{i(\sum \tau_j \nu^j)t},
\]
so $|\ft{\lambda}(n)| < \delta$ for all $n \neq 0$. Set
\[
K' = \{t : |X(t) - 1| \leq 3 \delta\},
\]
and use Lemma \ref{lemma:lambda-concentration-c0} to choose $N$ such
that
\[
\int_{\T \setminus K'} \lambda_s(t) \, \frac{dt}{2\pi} =
\lambda_s \{t \in \T : |X(t) - 1| > 3 \delta\} < \frac{\delta}{\|\rho\|_M}
\]
for every $s \in \Big( \frac1{2} - \delta, \frac1{2} \Big)$.
The function $h = \lambda \cdot \1_{K'}$ is then supported by $K'$, and
\[
\| \, \ft{\lambda} - \ft{h} \, \|_{\infty} \leq \|\lambda - h\|_{L^1(\T)}
= \|\lambda\|_{L^1(\T \setminus K')}
\leq \int \|\lambda_s\|_{L^1(\T \setminus K')} \, |d\rho(s)| < \delta.
\]
If we choose $\delta$ so that $\delta / (1 - \delta) < \eps$, then 
the function $g(t) = h(t) / \ft{h}(0)$ satisfies
\[
\ft{g}(0) = 1 \qquad \text{and} \qquad
|\ft{g}(n)| < \eps \quad (n \neq 0).
\]
Finally, to obtain a smooth function $f$ we take the convolution of $g$
with a $C^\infty$ non-negative function $\psi$ with integral $= 1$. The
condition $3 \delta < \eps$ will ensure that, if $\psi$ is supported by
a sufficiently small neighbourhood of $0$, then the support of $f$ will be
contained in the compact $K$ from \eqref{eq:compact-k-def-co}.
\end{proof}


\subsection{}
To finish the proof of Theorem \ref{thm:piatetski-theorem-c0} we proceed
in a similar way to \cite{kahane-salem}. For each $j=1,2,\dots$ we define
a number $\eps_j > 0$ and a positive integer $\nu_j \geq 3$, and choose
$f_j$, $X_j$ and $K_j$ according to Lemma \ref{lemma:principal-c0} (with
$\eps = \eps_j$ and $\nu = \nu_j$). We choose the $\{\eps_j\}$ by induction
to satisfy
\[
\eps_1 < 2^{-2} \qquad \textrm{and} \qquad
\|f_1 \cdot f_2 \cdot \ldots \cdot f_j\|_{A} \; \eps_{j+1} <
2^{-2-j} \quad (j=1,2,\dots).
\]
As in \cite[p. 215]{kahane-salem}, this implies that the product
$\prod_{j=1}^{\infty} f_j$ will converge in the space $A_\infty$ to
a nonzero distribution $S$ with Fourier coefficients tending to zero.
Denote $K = \bigcap_{j=1}^{\infty} K_j$, then clearly $S$ is supported
by $K$. On the other hand, we can choose $\{\nu_j\}$ to satisfy $\nu_j
\to \infty$. Since $\sum |\ft{X}_j(n)| \leq 2$, we deduce by Lemma
\ref{lemma:nomeasure-c0} that $K$ supports no nonzero measure with
Fourier coefficients tending to zero.

This concludes the proof of Theorem \ref{thm:piatetski-theorem-c0}.


\chapter{$\ell_q$ Spaces}
\label{chap:lq}

\section{Introduction}

\subsection{}
Chapter \ref{chap:c0} was concerned with Piatetski-Shapiro's theorem, which
states that there exists a compact $K$ on the circle, which supports a
distribution $S$ such that $\ft{S} \in c_0$, but which does not support
such a measure. We indicated this by saying that Piatetski-Shapiro's phenomenon
exists in the space $c_0$ of sequences tending to zero.

It is known from potential theory that no Piatetski-Shapiro phenomenon
exists in certain weighted $\ell_2$ spaces. Precisely, if a compact $K$
supports a distribution $S$ such that
$\sum |n|^{- \alpha} |\ft{S}(n)|^2 < \infty$, for some $0 < \alpha \leq 1$,
then $K$ also supports a positive measure satisfying this condition (see
\cite{kahane-salem}, Chapter III).

In this chapter we study Piatetski-Shapiro's phenomenon in $\ell_q$ spaces.
Only the case $q > 2$ is non-trivial, since only in this case there are
distributions $S$ such that $\ft{S} \in \ell_q$, but which are not measures.
Our main result is that Piatetski-Shapiro's phenomenon exists in $\ell_q$
for any $q > 2$.

\begin{theorem}
\label{thm:piatetski-theorem-lq}
For any $q > 2$ there is a compact $K$ on the circle, which supports
a distribution $S$ with $\ft{S} \in \ell_q$, but does not support such
a measure.
\end{theorem}

The main property of the compact $K$ in Theorem \ref{thm:piatetski-theorem-lq}
can also be formulated in the language of uniqueness theory. Namely:

\emph{For any $q > 2$ there is a compact $K$ such that
\begin{enumerate-math}
\item
There is a non-zero trigonometric series $\sum c(n) e^{int}$ with coefficients
$\{c(n)\}$ in $\ell_q$, converging to zero everywhere outside $K$.
\item
No Fourier-Stieltjes series may satisfy this property.
\end{enumerate-math}}

The result was published in \cite{lev-olevskii:piatetski}.

\subsection{}
In what follows $q$ is a fixed number, $q > 2$, and let $p = q/(q-1)$ denote the
exponent conjugate of $q$. We denote by $A_q(\T)$ the Banach space of Schwartz
distributions $S$ on the circle, endowed with the norm
$\|S\|_{A_q} := \|\ft{S}\|_{\ell_q}$. We also denote by $M(K)$ the space of finite
(complex) Borel measures supported by $K$, endowed with the total variation norm.


\section{Excluding measures}
We start by presenting a condition, analogous to that of Lemma \ref{lemma:nomeasure-c0},
which implies that a compact $K$ does not support any non-zero measure $\mu \in A_q$.

\begin{lemma}
\label{lemma:nomeasure-lq}
Suppose a compact $K$ on the circle satisfies the following
condition: for any positive integer $\nu$ there is a real
trigonometric polynomial
\[
X(t) = \sum_{|n| \geq \nu} \ft{X}(n) e^{int},
\quad \|\ft{X}\|_p \leq 1, \quad
\text{$\tfrac1{100} \leq X(t) \leq 100$ on $K$}.
\]
Then $K$ does not support a non-zero measure $\mu \in A_q$.
\end{lemma}

In the proof of Lemma \ref{lemma:nomeasure-lq} we will use:

\begin{lemma}
\label{lemma:measure-abs-lq-approx}
Let $\mu \in A_q$ be a measure supported by a compact $K$. Then the
measure $|\mu|$ belongs to the closure of $A_q \cap M(K)$ in the
$M(K)$ norm.
\end{lemma}

\begin{proof}
There exists a Borel function $\phi: \T \to \C$ with $|\phi(t)|=1$,
such that $|d\mu| = \phi \, d\mu$. Given $\eps > 0$, the theorem of Lusin
(see \cite{rudin}, Theorem 2.24) allows us to choose a continuous function
$\xi: \T \to \C$, $\|\xi\|_{C(\T)} \leq 1$ such that
\[
|\mu| \, \{t \in \T : \phi(t) \neq \xi(t) \} < \eps.
\]
Choose a trigonometric polynomial $\psi$ such that $\|\psi - \xi\|_{C(\T)} < \eps$,
then
\[
|\mu| \, \{t \in \T : |\phi(t) - \psi(t)| > \eps \} < \eps
\]
and $\|\psi\|_{C(\T)} < 2$. Thus the measure $d\mu_1 = \psi \, d\mu$ belongs to
$A_q \cap M(K)$, and
\[
\big\| \; |\mu| -  \mu_1 \; \big\|_{M(K)} = \int_{\T} |\phi - \psi|
\; |d\mu| \leq 3 \, \eps + \eps \, \|\mu\|_{M(K)}.
\qedhere
\]
\end{proof}

\begin{proof}[Proof of Lemma \ref{lemma:nomeasure-lq}]
Suppose that $\mu \in A_q \cap M(K)$. Given $\eps > 0$, by Lemma
\ref{lemma:measure-abs-lq-approx} one can find a measure $\mu_1
\in A_q \cap M(K)$ such that
\[
\big\| \; |\mu| -  \mu_1 \; \big\|_{M(K)} < \eps.
\]
Then
\[
\Big| \int_{\T} X d\mu_1 \Big| = 
\Big| \int_{K} X d\mu_1 \Big| \geq
\int_{K} X |d\mu| - 100 \eps \geq
\tfrac1{100} \|\mu\|_{M(K)} - 100 \eps.
\]
On the other hand,
\[
\Big| \int_{\T} X d\mu_1 \Big| = 
\Big| \sum_{|n| \geq \nu} \ft{X}(n) \ft{\mu}_1(-n) \Big| \leq
\Big(\sum_{|n| \geq \nu} |\ft{\mu}_1(n)|^q\Big)^{1/q} < \eps
\]
for sufficiently large $\nu$. Hence $\mu = 0$.
\end{proof}

\begin{remark}
The reason we need Lemma \ref{lemma:measure-abs-lq-approx} is that given a
measure $\mu \in A_q$, it is \emph{not} necessarily true that also $|\mu|
\in A_q$. One should compare this with the case when instead of $\ell_q$
we have the space $c_0$; by the lemma of Milicer-Gru\u{z}ewska, if the
Fourier coefficients of $\mu$ are tending to zero then also $|\mu|$ has
this property (see \cite{kechris-louveau}, II.5, Lemma 4).
\end{remark}


\section{Probabilistic tools}

\subsection{Exponential concentration}

In Lemma \ref{lemma:lambda-concentration-c0} the Che\-by\-shev inequality
was used to obtain concentration of a Riesz product probability measure.
In the current context a stronger estimate will be needed. We shall
use the following classical exponential estimate of S. N. Bernstein.
For completeness we also provide its proof.

\begin{lemma}
\label{lemma:bernstein-concentration}
Let $X_1, \dots, X_N$ be independent real random variables such that
$\E X_j = 0$ and $|X_j| \leq M_j$ almost surely, for each $1 \leq j
\leq N$. Then
\[
\prob \Big( \big| \sum_{j=1}^{N} X_j \big| > \alpha \Big) \leq
2 \exp \Big( - \tfrac1{2} \, \alpha^2 / \sum M_j^2 \Big).
\]
for every $\alpha > 0$.
\end{lemma}

\begin{proof}
Suppose first that $X$ is a real random variable, $\E X = 0$ and
$|X| \leq 1$ almost surely. Given any $\lambda \in \R$, by convexity
we have $e^{\lambda x} \leq \cosh \lambda + x \; \sinh \lambda$
for $-1 \leq x \leq 1$. Hence
\begin{equation}
\label{eq:subgaussian-estimate}
\E e^{\lambda X} \leq \cosh \lambda = \sum_{k=0}^{\infty} \frac{\lambda^{2k}}{(2k)!}
\leq \sum_{k=0}^{\infty} \frac{\lambda^{2k}}{2^k k!} = e^{\lambda^2 / 2}.
\end{equation}
Now apply \eqref{eq:subgaussian-estimate} to $M_j^{-1} X_j$ ($1 \leq j \leq N$),
then using Markov's inequality and the independence we get
\[
\begin{aligned}
\prob \Big( \sum_{j=1}^{N} X_j > \alpha \Big)
&= \prob \Big( \prod_{j=1}^{N} e^{\lambda X_j} > e^{\lambda \alpha} \Big)
\leq e^{-\lambda \alpha} \; \E \prod_{j=1}^{N} e^{\lambda X_j}\\	
&= e^{-\lambda \alpha} \; \prod_{j=1}^{N} \E e^{\lambda X_j}
\leq \exp \Big( \tfrac1{2} \, \lambda^2 \sum M_j^2 - \alpha \lambda \Big).
\end{aligned}
\]
Choosing $\lambda = \alpha / \sum M_j^2$ we obtain the one sided estimate
\[
\prob \Big( \sum_{j=1}^{N} X_j > \alpha \Big) \leq
\exp \Big( - \tfrac1{2} \, \alpha^2 / \sum M_j^2 \Big).
\]
The two-sided estimate can now be easily deduced by symmetry.
\end{proof}

It is now easy to pass to the case of non-zero expectation; we omit
the proof.

\begin{lemma}
\label{lemma:bernstein-hoeffding-concentration}
Let $X_1, \dots, X_N$ be independent real random variables such that
$|X_j| \leq M$ almost surely, for each $1 \leq j \leq N$.
Denote $S = \sum_{j=1}^{N} X_j$, then
\begin{equation}
\label{eq:bernstein-hoeffding}
\prob \big\{ |S - \E S| > \alpha \big\} \leq 2 \exp
\Big(- \frac{\alpha^2} {8 M^2 N} \Big)
\end{equation}
for every $\alpha > 0$.
\end{lemma}

\begin{remark}
The constant $1/8$ in Lemma \ref{lemma:bernstein-hoeffding-concentration}
is not optimal, see Hoeffding \cite[Theorem 2]{hoeffding}. For further results
in this context we refer to Petrov \cite{petrov}.
\end{remark}



\subsection{Stochastic independence}

Let $g$ denote a real $2\pi$-periodic function on $\R$. Given such
$g$, consider the system of functions
\[
\{g(\nu^j t)\}_{j=1}^{N}
\]
on the circle $\T$. We are interested in stochastic properties of
this system.

\begin{lemma}
\label{lemma:stochastic-independence-lebesgue}
Suppose that $g$ is constant on each interval
\begin{equation}
\label{eq:nu-intervals}
I^{(\nu)}_k = \big(\tfrac{2 \pi (k-1)}{\nu}, \tfrac{2 \pi k}{\nu}
\big), \quad 1 \leq k \leq \nu.
\end{equation}
Then the system $\{g(\nu^j t)\}_{j=1}^{N}$ is stochastically independent
on the circle $\T$, with respect to the Lebesgue measure.
\end{lemma}

\begin{proof}
Each of the functions $g(\nu^j t)$, $1 \leq j \leq N-1$, is constant on
each interval $\big(2 \pi (k-1) \nu^{-N}, 2 \pi k \nu^{-N}\big)$, where
$1 \leq k \leq \nu^N$. On the other hand, the function
$g(\nu^N t)$ has the same distribution, with respect to the Lebesgue
measure, on each one of these intervals. To finish the proof one 
continues by induction.
\end{proof}

We next consider stochastic independence of the system
$\{g(\nu^j t)\}_{j=1}^{N}$ not with respect to the Lebesgue measure,
but rather with respect to a certain Riesz product type measure based
on $g$. Precisely, consider the measure
\begin{equation}
\label{eq:g-product-measure}
\prod_{j=1}^{N} \big(1 + r_j \, g(\nu^j t)\big) \,
\frac{dt}{2\pi}, \qquad -1 < r_j < 1.
\end{equation}
We assume that $g$ is constant on each interval
\eqref{eq:nu-intervals}, and satisfies the conditions
\begin{equation}
\label{eq:g-conditions}
-1 \leq g(t) \leq 1, \quad \int_{0}^{2\pi} g(t) \, dt = 0.
\end{equation}
The measure \eqref{eq:g-product-measure} is therefore positive,
and using \eqref{eq:g-conditions} and Lemma
\ref{lemma:stochastic-independence-lebesgue},
\[
\int_{\T} \; \prod_{j=1}^{N} \big(1 + r_j \, g(\nu^j t)\big) \;
\frac{dt}{2\pi} =
\prod_{j=1}^{N} \int_{\T} \big(1 + r_j \, g(\nu^j t)\big) \,
\frac{dt}{2\pi} = 1.
\]
That is, \eqref{eq:g-product-measure} defines a probability
measure on the circle $\T$. We now have

\begin{lemma}
\label{lemma:stochastic-independence-product}
Let $g$ be constant on each interval \eqref{eq:nu-intervals}
and satisfy \eqref{eq:g-conditions}. Then the system
$\{g(\nu^j t)\}_{j=1}^{N}$ is stochastically independent on the
circle $\T$ with respect to the probability measure
\eqref{eq:g-product-measure}.
\end{lemma}

\begin{proof}
Consider $Y_j(t) = g(\nu^j t)$ as random variables with respect to the
probability measure \eqref{eq:g-product-measure}. Let $\xi_1, \dots, \xi_N
\in \R$ be given. Using the independence with respect to the Lebesgue measure
as in Lemma \ref{lemma:stochastic-independence-lebesgue}, one has
\[
\begin{aligned}
\E \exp \Big( i \sum_{j=1}^{N} \xi_j Y_j \Big)
&= \int_{\T} \prod_{j=1}^{N} 
\big(1 + 2 \, r_j \, g(\nu^j t) \big) \;
e^{i \xi_j g(\nu^j t)} \; \frac{dt}{2\pi}\\
&= \prod_{j=1}^{N} \int_{\T}
\big(1 + 2 \, r_j \, g(t) \big) \;
e^{i \xi_j g(t)} \; \frac{dt}{2\pi}.
\end{aligned}
\]
Hence, denoting $\varphi_j(\xi) = \int_{\T} \big(1 + 2 \, r_j \, g(t) \big)
\, e^{i \xi g(t)} \, \frac{dt}{2\pi}$ $(1 \leq j \leq N)$, it is seen that
\[
\E \exp \Big( i \sum_{j=1}^{N} \xi_j Y_j \Big)
= \varphi_1(\xi_1) \, \varphi_2(\xi_2) \, \cdots \, \varphi_N(\xi_N).
\]
By properties of the Fourier transform, we conclude that the
distribution in $\R^N$ of $(Y_1, \dots, Y_N)$ is a product
measure, which proves the independence.
\end{proof}


\section{Riesz products}

\subsection{Notation}
Define a Riesz product
\[
\lambda_s(t) = \prod_{j=1}^{N} \big(1 +
2 s  N^{-1/q} \cos \nu^j t \big).
\]
Here we assume that $\lowers < s < \uppers$ and $\nu \geq 3$. Expanding
the product gives
\begin{equation}
\label{eq:lambda-trig-expansion-lq}
\lambda_s(t) = 1 + \sum_{\tau \neq 0} (s N^{-1/q})^{\sum
|\tau_j|} \, e^{i (\tau_1 \nu + \tau_2 \nu^2 + \dots + \tau_N \nu^N) t},
\end{equation}
where the sum is over all non-zero vectors
\[
\tau = (\tau_1, \dots, \tau_N), \quad \tau_j \in \{-1,0,1\}.
\]
As before, the frequencies in the sum \eqref{eq:lambda-trig-expansion-lq}
are distinct.


\subsection{``Almost independence''}
As before, we also use $\lambda_s$ to denote the probability measure
$\lambda_s(t) dt / 2\pi$ on the circle. Define a trigonometric
polynomial
\begin{equation}
\label{eq:x-cosines-lq}
X(t) = N^{-1/p} \sum_{j=1}^{N} \cos \nu^j t.
\end{equation}

If $\nu$ is sufficiently large, the members of the polynomial $X$
are ``almost independent'' with respect to the measure $\lambda_s$.
Precisely, we can prove the following exponential concentration estimate.

%

\begin{lemma}
\label{lemma:lambda-concentration-lq}
Being given $N$, for $\nu \geq \nu(N)$ one has 
\begin{equation}
\label{eq:lambda-concentration-lq}
\lambda_s \{t \in \T : |X(t)-s| > \alpha\} 
\leq 3 \exp\big(- \tfrac1{8} \alpha^2 N^{2/p \; - \; 1} \big)
\end{equation}
for every $\alpha > 0$ and every $\lowers < s < \uppers$.
\end{lemma}

\begin{proof}
Define a $2\pi$-periodic function $g$, constant on each interval
\eqref{eq:nu-intervals}, by the requirement
\[
\int_{I^{(\nu)}_k} g(t) \; dt = 
\int_{I^{(\nu)}_k} \cos t \; dt.
\]
Then automatically also \eqref{eq:g-conditions} is satisfied.
Denote $\delta = \max_{t \in \T} |g(t)-\cos t|$, then clearly
$\delta \to 0$ as $\nu \to \infty$. Consider the probability measure
\[
\gamma_s(t) = \prod_{j=1}^{N} \big(1 + 2 s N^{-1/q} \,
g(\nu^j t)\big) \; \frac{dt}{2 \pi}.
\]
Denote $Y = N^{-1/p} \sum_{j=1}^{N} g(\nu^j t)$, and let
$\eta = \int_{\T} Y(t)\, \gamma_s(t) \, \frac{dt}{2\pi}$
be the expectation of $Y$ with respect to $\gamma_s$.
Using Lemma \ref{lemma:stochastic-independence-lebesgue}
and the properties of the function $g$, one can compute that
\[
\eta = 2 s \int_{\T} g(t)^2 \, \frac{dt}{2\pi}.
\]
We now use the following fact: the function $g(t)$, viewed as an element
of $L^2(\T)$, is obtain by orthogonal projection of $\cos t$ on the subspace
of functions constant on each interval \eqref{eq:nu-intervals}. Thus, by
the Pythagorean theorem
\[
s - \eta = 2s \int_{\T} \cos^2 t \; \frac{dt}{2\pi}
- 2s \int_{\T} g(t)^2 \; \frac{dt}{2\pi}
= 2s \int_{\T} (\cos t - g(t))^2 \; \frac{dt}{2\pi} \, ,
\]
which implies $|\eta - s| \leq \delta^2$. Since we also have $|X(t) - Y(t)| \leq 
\delta \, N^{1/q}$, we obtain
\[
|X(t) - s| \leq |Y(t) - \eta| + \delta \, N^{1/q} + \delta^2.
\]
Also, for $\lowers < s < \uppers$ we have
\[
\begin{aligned}
\frac{1 + 2 \, s \, N^{-1/q} \, \cos t}{1 + 2 \, s \, N^{-1/q} \, g(t)}
&= 1 + \frac{2 \, s \, N^{-1/q} \, (\cos t - g(t))}{1 + 2 \, s \, N^{-1/q} \, g(t)}\\
&\leq \exp \Big( \frac{2 \, s \, N^{-1/q} \, \delta}{1 - 2 \, s \, N^{-1/q}} \Big)
\leq \exp (2 \, \delta \, N^{-1/q}),
\end{aligned}
\]
hence $\lambda_s(t) \leq \gamma_s(t) \exp(2 \, \delta \, N^{1/p})$.
Setting
\[
\eps = \max \left\{\exp(2 \delta N^{-1/q}) - 1, \; \delta \, N^{1/q} + \delta^2 \right\},
\]
we have therefore proved that
\begin{equation}
\label{eq:step-approx-estimates}
|X(t) - s| \leq |Y(t) - \eta| + \eps \quad \text{and} \quad
\lambda_s(t) \leq (1 + \eps) \, \gamma_s(t).
\end{equation}
Now using \eqref{eq:step-approx-estimates}, Lemma
\ref{lemma:stochastic-independence-product}
and Lemma \ref{lemma:bernstein-hoeffding-concentration} we deduce the estimate
\begin{align}
\lambda_s \{t : |X(t) - s| > \alpha\} &\leq
(1 + \eps) \, \gamma_s \{t : |Y(t) - \eta| > \alpha - \eps \}\notag\\
\label{eq:step-approx-probability}
& \leq 2 (1 + \eps) \exp\big(- \tfrac1{8} (\alpha - \eps)^2 N^{2/p \; - \; 1} \big).
\end{align}

Note that, since $|X(t)| \leq N^{1/q}$ and $\lowers < s < \uppers$, the left side of
\eqref{eq:lambda-concentration-lq} is zero if $\alpha > N^{1/q} + 1$. It is
therefore enough to obtain \eqref{eq:lambda-concentration-lq} for $\alpha \leq
N^{1/q} + 1$. However, a straightforward calculation shows that this will follows
from \eqref{eq:step-approx-probability} if $\eps$ is chosen sufficiently small, in
a way which depends only on $N$. Since choosing $\nu$ sufficiently large will make
$\eps$ arbitrarily small, this proves the claim.
\end{proof}

\subsection{Concentration in $L^2$}

We use the probabilistic exponential concentration to prove:

\begin{lemma}
\label{lemma:lambda-restriction-lq}
Denote
\[
K' = \big\{ t \in \T \; : \; \tfrac1{90} \leq X(t) \leq
90 \big\}.
\]
Given $\delta > 0$, for any $N \geq N(\delta)$ and any
$\nu \geq \nu(N)$ one has
\[
\int_{\T \setminus K'} \lambda_s^2(t) \; \frac{dt}{2\pi} < \delta
\]
for every $\lowers < s < \uppers$.
\end{lemma}

\begin{proof}
We have
\[
\lambda_s(t) = \prod_{j=1}^{N} \big(1 +  2 \, s \, N^{-1/q} \cos
\nu^j t \big) \leq \exp \Big( 2s \, N^{-1/q} \sum_{j=1}^{N}
\cos \nu^j t \Big),
\]
so using \eqref{eq:x-cosines-lq} we obtain an estimate
\begin{equation}
\label{eq:lambda-max-lq}
\lambda_s(t) \leq \exp \left( 2 \, s \, N^{2/p \; - \; 1} \, X(t)
\right).
\end{equation}
Now Lemma \ref{lemma:lambda-concentration-lq} implies that 
\eqref{eq:lambda-concentration-lq} holds for any $\nu \geq \nu(N)$.
Combining with \eqref{eq:lambda-max-lq}, it follows that for
every $\lowers < s < \uppers$
\[
\begin{aligned}
\int_{\big\{t \; : \; X(t) < \frac1{90}\big\}} \lambda_s^2(t) \; \frac{dt} {2\pi}
&\leq \int_{\big\{t \; : \; X(t) < \frac1{90}\big\}} \lambda_s(t) \;
\frac{dt}{2\pi} \cdot \max_{\big\{t \; : \; X(t) < \frac1{90}\big\}}
\lambda_s(t)\\[4pt]
&\leq 3 \exp \left( -\tfrac1{8}(s - \tfrac1{90})^2 N^{2/p \; - \; 1}
\right) \cdot \exp \left(2 \, s \, N^{2/p \; - \; 1} \cdot \tfrac{1}{90}
\right)\\[6pt]
&\leq 3 \exp \left(- 2^{-10} N^{2/p \; - \; 1} \right),
\end{aligned}
\]
and, for any integer $90 \leq k \leq N^{1/q}$,
\[
\begin{aligned}
&\int_{\big\{t \; : \; k < X(t) \leq k+1\big\}} \lambda_s^2(t) \;
\frac{dt}{2\pi}
\leq \int_{\big\{t \; : \; k < X(t) \leq k+1\big\}} \lambda_s(t) \;
\frac{dt}{2\pi}
\cdot \max_{\big\{t \; : \; k < X(t) \leq k+1\big\}} \lambda_s(t)\\[4pt]
&\qquad \qquad \leq 3 \exp \left( -\tfrac1{8}(k - s)^2 N^{2/p \; - \; 1} \right) \cdot
\exp \left(2 \, s \, N^{2/p \; - \; 1} (k+1) \right)\\[6pt]
&\qquad \qquad \leq 3 \exp \left(- \, N^{2/p \; - \; 1} \right).
\end{aligned}
\]
Hence, keeping in mind that $X(t) \leq N^{1/q}$ for every $t$,
we obtain for $\nu \geq \nu(N)$
\[
\int_{\T \setminus K'} \lambda_s^2(t) \; \frac{dt}{2\pi}
\leq 3 \exp \left(- 2^{-10} N^{2/p \; - \; 1} \right) +
3 N^{1/q} \exp \left(- N^{2/p \; - \; 1} \right)
\xrightarrow[N \to \infty]{} 0
\]
uniformly in $s$.
\end{proof}


\section{Proof of main result}

In this section we finish the proof of Theorem \ref{thm:piatetski-theorem-lq}.

\subsection{}
We start with an analog of Lemma \ref{lemma:principal-c0}.

\begin{lemma}
\label{lemma:principal-lq}
Let $\eps > 0$. For $N \geq N(\eps)$ and $\nu \geq \nu(N)$, 
there exist a $C^\infty$ function $f: \T \to \C$ such that
\begin{enumerate-math}
\item
$f$ is supported by
\begin{equation}
\label{eq:k-lq}
K = \big\{ t \in \T \; : \; \tfrac1{100} \leq X(t) \leq 100\big\},
\end{equation}
where $X$ is defined by \eqref{eq:x-cosines-lq}.
\item
$f(t) = 1 + \sum_{n \neq 0} \ft{f}(n)e^{int}$,
where $\Big( \sum_{n \neq 0} |\ft{f}(n)|^q \Big)^{1/q} < \eps$.
\end{enumerate-math}
\end{lemma}

\begin{proof}
Given $\delta > 0$, according to Lemma \ref{lemma:kahane-lemma} there is
a measure $\rho$ supported by the interval $(\lowers, \uppers)$
such that
\[
\int d\rho = 1 \qquad \text{and} \qquad
\Big| \int s^k \, d\rho(s) \Big| < \delta
\quad (k=1,2,\dots).
\]
Define
\[
\lambda(t) = \int \lambda_s(t) \, d\rho(s).
\]
By \eqref{eq:lambda-trig-expansion-lq}, the Fourier expansion of $\lambda$ is
\[
\lambda(t) = 1 + \sum_{\tau \neq 0}
\Big\{ N^{-1/q \sum |\tau_j|} \int s^{\sum |\tau_j|} d\rho(s) \Big\}
\; e^{i (\sum \tau_j \nu^j) t}.
\]
It follows that
\[
\sum_{n \neq 0} \big|\ft{\lambda}(n)\big|^q
< \delta^q \sum_{\tau \in \{-1,0,1\}^N}
N^{-\sum |\tau_j|}
= \delta^q \, \Big(1 + \frac{2}{N} \Big)^N < e^2 \, \delta^q.
\]
Now use Lemma \ref{lemma:lambda-restriction-lq} to choose $N$ such that
\[
\left( \int_{\T \setminus K'} \lambda_s^2(t) \; \frac{dt}{2\pi}
\right)^{1/2} < \frac{\delta}{\|\rho\|_M}
\]
for any $\lowers < s < \uppers$. Then $h := \lambda \cdot \1_{K'}$
is supported by $K'$, and
\[
\|\lambda - h\|_{A_q} \leq \|\lambda - h\|_{L^2(\T)}
= \|\lambda\|_{L^2(\T \setminus K')}
\leq \int \|\lambda_s\|_{L^2(\T \setminus K')} \; |d\rho(s)|
< \delta.
\]
If $\delta = \delta(\eps)$ is chosen sufficiently small, the
function $g(t) = h(t) / \ft{f}_1(0)$ satisfies
\[
\ft{g}(0) = 1 \qquad \text{and} \qquad
\Big( \sum_{n \neq 0} |\ft{g}(n)|^q \Big)^{1/q} < \eps.
\]
To obtain the smooth function $f$, as before, we take the convolution
of $g$ with a $C^\infty$ non-negative function with integral $= 1$,
which is supported on a sufficiently small neighborhood of
$0$ to ensure that $f$ will be supported by $K$.
\end{proof}

\subsection{}
We now conclude the proof with a procedure similar to \cite[p. 215]{kahane-salem}.

\begin{proof}[Proof of Theorem \ref{thm:piatetski-theorem-lq}]
For a sequence $\{\eps_j\}$ let $f_j$, $X_j$ and $K_j$ be given by Lemma
\ref{lemma:principal-lq}, for an appropriate choice of $N_j = N(\eps_j)$
and $\nu_j = \nu(\eps_j)$. We choose $\{\eps_j\}$ by induction to satisfy
\[
\eps_1 < 2^{-2} \qquad \textrm{and} \qquad
\|f_1 \cdot f_2 \cdot \ldots \cdot f_j\|_{A} \; \eps_{j+1} <
2^{-2-j} \quad (j=1,2,\dots).
\]
This implies that the product $\prod_{j=1}^{\infty} f_j$ will converge
in the $A_q$ norm to a non-zero distribution $S \in A_q$. Indeed, consider
the partial products $S_0 = 1$ and $S_j = f_1 \cdots f_j$, then
\[
\|S_{j+1} - S_j\|_{A_q} = \|f_1 \cdots f_j \cdot (f_{j+1} - 1)\|_{A_q} \leq
\|f_1 \cdots f_j\|_{A} \; \eps_{j+1} < 2^{-2-j},
\]
hence $S_j$ converges in $A_q$ to a limit $S$. We have
\[
\|S - 1\|_{A_q} \leq \sum_{j=0}^{\infty} \|S_{j+1} - S_j\|_{A_q} \leq
\sum_{j=0}^{\infty} 2^{-2-j} < 1
\]
and so $S \neq 0$. Denote $K = \bigcap_{j=1}^{\infty} K_j$, then clearly $S$ is
supported on $K$. We may assume $\nu_j \to \infty$, and we have $\|\ft{X}_j\|_p
\leq 1$, so Lemma \ref{lemma:nomeasure-lq} implies that $K$ supports no non-zero
measure in $A_q$. This completes the proof of Theorem \ref{thm:piatetski-theorem-lq}.
\end{proof}


\chapter{Generators in $\ell_p$ and Zero Set of Fourier Transform}
\label{chap:generators-lp}

\section{Introduction}

\subsection{}

A function $F: \Z \to \C$ is called a \emph{generator}, or a \emph{cyclic vector},
in the space $\ell_p(\Z)$ if the linear span of its translates is dense in the space.
In other words, $F \in \ell_p(\Z)$ is a generator if the closed translation
invariant linear subspace generated by $F$ is the whole $\ell_p$. How to know whether
a given $F$ is a generator, or not? For $p=1$ and $2$, Wiener characterized the generators
by the zero set of the Fourier transform
\[
f(t) := \sum_{n \in \Z} F(n) \, e^{int},
\]
as follows:
\begin{quote}
\emph{$F$ is a generator in $\ell_1$ if and only if $f(t)$ has no zeros.\\[4pt]
$F$ is a generator in $\ell_2$ if and only if $f(t) \neq 0$ almost everywhere.}
\end{quote}
The same characterization is true for $L^1(\R)$ and $L^2(\R)$, see \cite{wiener}.

``Interpolating'' between $p=1$ and $2$ one may expect that the generators
in $\ell_p$ (or $L^p$), $1 < p < 2$, could be characterized by the condition
that the zero set of the Fourier transform is ``small'' in a certain sense.
In this context various metrical, arithmetical and other properties of the
zero set for generators and non-generators have been studied by Beurling
\cite{beurling}, Pollard \cite{pollard}, Herz \cite{herz}, Newman \cite{newman}
and other authors (see \cite{segal:span}, \cite{edwards}, \cite{kinukawa}). However,
none of these results provides a complete characterization of the generators.

For example, Beurling proved in \cite{beurling} the following result:
\begin{quote}
\emph{Let $F \in \ell_1$. If the Hausdorff dimension of the zero set
$Z_f$ is less than $2 - 2/p$ then $F$ is a generator in $\ell_p$ $(1<p<2)$.}
\end{quote}
The converse, however, is not true. Indeed, one can construct $F$ which is
a generator in every $\ell_p$, $1 < p < 2$, but $Z_f$ has Hausdorff
dimension $1$ (see \cite{newman}).

\subsection{}
In this chapter we study the problem: \emph{is it possible at all} to characterize
the generators in $\ell_p$ $(1 < p < 2$) by the zero set of the Fourier transform?

It should be pointed out that for $1 < p < 2$, there is no canonical way to define
the zero set of the Fourier transform of a general element in $\ell_p$. First we
note that no such a problem arises in the cases $p=1,2$. Indeed, if $F \in \ell_1$
then the zero set $Z_f$ is a well-defined closed set, since then $f$ is a continuous
function. If $F \in \ell_2$ then $f \in L^2(\T)$, so in this case the set $Z_f$ is
defined up to measure zero. In the case $1 < p < 2$, one may notice that the set $Z_f$
could also be defined up to measure zero, since by the Hausdorff-Young theorem
$f \in L^q(\T)$, where $q = p/(p-1)$. However, this approach cannot provide a
characterization of the generators by the zero set: there exists $F$ which is not a
generator in $\ell_p$ $(1 < p < 2)$, but $f(t) \neq 0$ almost everywhere (see \cite{segal:span}).

Our main result shows that, unlike the classical cases $p=1$ and $2$, \emph{the characterization
of the generators in $\ell_p$ $(1 < p < 2)$ by the zero set $Z_f$ of the Fourier transform
is impossible}. Precisely, we prove the following theorem:

\begin{theorem}
\label{thm:impossible-integer}
Given $1 < p < 2$ one can find two continuous functions $f$ and $g$
on the circle $\T$, with the following properties:
\begin{enumerate-math}
\item $\{t : f(t) = 0\} = \{t : g(t) = 0\}$,
\item $F := \ft{f}$ and $G := \ft{g}$ are both in $\ell_p(\Z)$,
\item $G$ is a generator in $\ell_p$, but $F$ is not.
\end{enumerate-math}
\end{theorem}

\begin{remarks}
\item
We emphasize the role of the continuity condition in the theorem, which makes precise the
concept of the ``zero set''. The result shows that it does not matter how one may define
the zero set in general; the characterization of the generators in $\ell_p$ $(1 < p < 2)$
by the zero set $Z_f$ is impossible already in the case when $Z_f$ is well-defined.

\item
As we will show, the function $f$ in Theorem \ref{thm:impossible-integer} can in fact be
chosen infinitely smooth. However, $f$ and $g$ cannot both be smooth. We discuss this in
some more detail below.

\item
One may check that Theorem \ref{thm:impossible-integer} is not true for $p > 2$.
However, in no way does this mean that the generators in $\ell_p$ $(p > 2)$ could be
characterized by the zero set of their Fourier transforms. In fact, the Fourier transform
of an element in $\ell_p$ $(p > 2)$ is generally not a function, but rather a distribution,
so one cannot even define its zero set in general.
\end{remarks}

The $L^p(\R)$ version of the result is also true:

\begin{theorem}
\label{thm:impossible-real}
Given $1 < p < 2$ one can find two functions $F$ and $G$ in $L^p(\R)$ with the
following properties. The Fourier transforms $\ft{F}(t)$, $\ft{G}(t)$ are
continuous functions on $\R$; they have the same zero set; the set of translates
$\{G(x - u)\}$, $u \in \R$, spans the whole space, but $\{F(x-u)\}$ does not.
\end{theorem}

\begin{remark}
As we show at the end of this chapter, the functions $F$ and $G$ in Theorem
\ref{thm:impossible-real} could be chosen infinitely smooth, and even as the
restrictions to the real line of two entire functions of order $1$.
\end{remark}

The results of this chapter were published in \cite{lev-olevskii:generators}.

\section{Reformulation}

\subsection{}
The Fourier transform
\[
F \in \ell_p(\Z) \; \mapsto \;
f(t) := \sum_{n \in \Z} F(n) \, e^{int}
\]
allows to identify $\ell_p(\Z)$ with the Banach space $A_p(\T)$ of functions on the
circle with Fourier coefficients in $\ell_p$, endowed with the norm
$\|f\|_{A_p} = \|\ft{f}\|_{\ell_p}$. A linear combination of translates of $F$
corresponds via this mapping to a multiplication of $f$ by a trigonometric
polynomial. It is therefore obvious that $F$ is a generator in $\ell_p(\Z)$ if
and only if the set $\{P(t) f(t)\}$, where $P$ goes through all trigonometric
polynomials, is dense in the space $A_p(\T)$.

Theorem \ref{thm:impossible-integer} could be therefore reformulated
in the following way: there exist two continuous functions $f, g \in A_p(\T)$,
having the same zero set, such that the set
\begin{equation}
\label{eq:g-ideal}
\{P(t) g(t) \; : \; \text{$P$ is a trigonometric polynomial}\}
\end{equation}
is dense in $A_p(\T)$, but the set $\{P(t) f(t)\}$ is not dense.

%

\subsection{}
We will prove the following theorem.

\begin{theorem}
\label{thm:strong-piatetski}
For any $1 < p < 2$ one can construct a compact $E \subset \T$,
and a function $g \in C(\T) \cap A_p(\T)$, such that:
\begin{enumerate-math-abc}
\item
\label{item:g-zeros-e}
$Z_g := \{t : g(t)=0\} = E$;
\item
\label{item:g-ideal-dense}
The set \eqref{eq:g-ideal} is dense in $A_p$;
\item
\label{item:e-supp-dist}
There is a (non-zero) distribution $S$, supported by $E$, which
belongs to $A_q$, where $q = p/(p - 1)$ is the conjugate of $p$.
\end{enumerate-math-abc}
\end{theorem}

As explained above, \ref{item:g-ideal-dense} means that $\ft{g}$ is a generator.
On the other hand \ref{item:e-supp-dist} is equivalent to the fact that no Fourier
transform of a smooth function $f$ vanishing on $E$, could be a generator. More
precisely: if $F = \ft{f} \in \ell_1$ is not a generator in $\ell_p$, then $Z_f$
must support a non-zero distribution $S \in A_q$ (see \cite{kahane-salem}, Chapter
III, \textsection 6); and conversely, if $f \in \lip(\alpha)$, $\alpha > 1/p - 1/2$,
and $Z_f$ supports such a distribution, then $F$ is not a generator in $\ell_p$
(see \cite{kahane-salem}, Chapter IX, \textsection 6).

Theorem \ref{thm:impossible-integer} is therefore a direct consequence of Theorem
\ref{thm:strong-piatetski}. Moreover, the function $f$ in Theorem
\ref{thm:impossible-integer} can be chosen infinitely smooth: any $f$ such that
$Z_f = E$ will do. Theorem \ref{thm:impossible-real} also follows from
Theorem \ref{thm:strong-piatetski}, as we show later on.

\subsection{}
Theorem \ref{thm:strong-piatetski} is a strengthening of our result from Chapter
\ref{chap:lq}, where we constructed a compact $K$ which supports a distribution
belonging to $A_q$ $(q > 2)$, but which does not support such a measure (Theorem
\ref{thm:piatetski-theorem-lq}). Indeed, we will now show that from
\ref{item:g-ideal-dense} it follows that $E$ cannot support a non-zero measure
$\mu \in A_q$, hence the compact $E$ from Theorem \ref{thm:strong-piatetski} also
satisfies the result of Theorem \ref{thm:piatetski-theorem-lq}.

To see this, suppose that $\mu \in A_q$ is a measure supported by $E$. Then we have
\[
\sum_{n \in \Z} \ft{g}(n - k) \, \ft{\mu}(-n) = \int_{\T} g(t) e^{ikt} d\mu(t) = 0,
\quad k \in \Z,
\]
since $\mu$ is supported by $Z_g$. The Fourier transform $\ft{\mu}$ is therefore
an element of $\ell_q$ (the dual space of $\ell_p$) which annihilates all the
translates of $\ft{g}$. But $\ft{g}$ is a generator in $\ell_p$, so we must have
$\ft{\mu} = 0$. Hence also $\mu = 0$.



%
%
%

%
%
%


\section{Riesz-type products}

\subsection{Notation}
As before we use finite Riesz product, but instead of the cosine function we now
use a certain trigonometric polynomial $\varphi$. We define
\begin{equation}
\label{eq:riesz-prod-generators}
\lambda_s(t) = \prod_{j=1}^{N} (1 + s \, \varphi(\nu^j t))
\end{equation}
where $N$ and $\nu$ denote positive integers, the parameter $s$
denotes a real number and $\varphi$ is a trigonometric polynomial.
We will assume that
\begin{equation}
\label{eq:assumptions-phi}
\text{$\varphi$ is real}, \quad \ft{\varphi}(0) = 0, \quad
\|\varphi\|_{\infty} \leq 1, \quad
\|\varphi\|_{L^2(\T)} > \tfrac{9}{10}.
\end{equation}

\subsection{Fourier expansion}
Replacing $\varphi$ by its Fourier series in \eqref{eq:riesz-prod-generators}
and expanding the product, we get
\begin{equation}
\label{eq:lambda-trig-expansion-generators}
\lambda_s(t) = 1 + \sum_{k \neq 0} \Big\{ s^{l(k)}
\prod_{k_j \neq 0} \ft{\varphi}(k_j) \Big\}
\; e^{i(k_1 \nu + k_2 \nu^2 + \cdots + k_N \nu^N) t},
\end{equation}
where the sum is over all non-zero vectors
\[
k = (k_1, \dots, k_N) \in \Z^N, \quad |k_j| \leq \deg \varphi,
\]
and $l(k) > 0$ is the number of non-zero coordinates of $k$.
%
We will assume the condition
\begin{equation}
\label{eq:nu-degree}
\nu > 2 \deg \varphi,
\end{equation}
which guarantees that every integer $n$ admits at most one representation
\[
n = k_1 \nu + k_2 \nu^2 + \cdots + k_N \nu^N
\]
with $k$ as above.
The members of the sum \eqref{eq:lambda-trig-expansion-generators} are therefore
exponentials with distinct frequencies, so \eqref{eq:lambda-trig-expansion-generators}
is the Fourier expansion of $\lambda_s$.

\subsection{Probabilistic concentration}
For $0 < s < 1$ it follows from \eqref{eq:riesz-prod-generators},
\eqref{eq:assumptions-phi} and \eqref{eq:lambda-trig-expansion-generators} that the
function $\lambda_s$ is positive and has integral $= 1$. We thus may
consider a probability measure $\lambda_s(t) \, {dt}/{2\pi}$ on the
circle. Define a trigonometric polynomial
\begin{equation}
\label{eq:def-x}
X(t) = \frac1{N} \sum_{j=1}^{N} \varphi(\nu^j t),
\end{equation}
and consider $X$ as a random variable with respect to the measure
$\lambda_s(t) \, {dt}/{2\pi}$.

\begin{lemma}
\label{lemma:concentration-l1}
For $\nu \geq \nu(N, \varphi)$ one has
\begin{equation}
\label{eq:prob-concentration}
\int_{\{t \in \T \; : \; |X(t) - \E(X)| > \alpha\}} \lambda_s(t) \, \frac{dt}{2\pi}
\leq 3 \exp ( - \tfrac1{8} \alpha^2 N)
\end{equation}
which holds for every $\alpha > 0$ and every $s \in \big(\tfrac{8}{10}, \frac{9}{10}\big)$.
\end{lemma}

The proof of Lemma \ref{lemma:concentration-l1} is very similar to the proof
of Lemma \ref{lemma:lambda-concentration-lq}, and is therefore omitted.

\subsection{Concentration in $L^2$}
Using \eqref{eq:lambda-trig-expansion-generators} and \eqref{eq:nu-degree} one can calculate
\begin{equation}
\label{eq:x-expectation}
\E (X) = \int_{\T} X(t) \, \lambda_s(t) \, \frac{dt}{2\pi} = s \|\varphi\|_{L^2(\T)}^2.
\end{equation}

\begin{lemma}
\label{lemma:concentration-l2}
Given $\delta > 0$ there is $N(\delta)$ with the following property.
For each $N \geq N(\delta)$ and each trigonometric polynomial $\varphi$
satisfying \eqref{eq:assumptions-phi} one can find $\nu = \nu(N, \varphi)$
satisfying \eqref{eq:nu-degree} such that
\begin{equation}
\label{eq:concentration-l2}
\int_{\{t \in \T \; : \; X(t) < \frac1{40} \}} \lambda_s^2(t) \, \frac{dt}{2\pi}
< \delta
\end{equation}
holds for every $s \in \big(\tfrac{8}{10}, \frac{9}{10}\big)$.
\end{lemma}

\begin{proof}
Given $N$ and $\varphi$, use Lemma \ref{lemma:concentration-l1} to choose
$\nu = \nu(N, \varphi)$. We have the estimate
\begin{equation}
\label{eq:estimate-lambda-x}
\lambda_s(t) \leq \exp \Big(s \sum_{j=1}^{N} \varphi(\nu^j t) \Big)
= \exp \Big(s N \, X(t) \Big).
\end{equation}
From \eqref{eq:x-expectation} it follows that $\E (X) > \tfrac{5}{8}$,
so using \eqref{eq:prob-concentration} and \eqref{eq:estimate-lambda-x}
we get
\[
\begin{aligned}
\int_{\{t \; : \; X < \frac1{40}\}} \lambda_s^2(t) \, \frac{dt}{2\pi} &\leq 
\bigg( \int_{\{t \; : \; X < \frac1{40}\}} \lambda_s(t) \, \frac{dt}{2\pi} \bigg)
\bigg( \sup_{\{t \; : \; X(t) < \frac1{40}\}} \lambda_s(t) \bigg)\\[4pt]
&\leq 3 \exp \big( - \tfrac1{8} (\tfrac{5}{8} - \tfrac1{40})^2 N \big)
\exp \big( s N \cdot \tfrac1{40} \big)\\[4pt]
&< 3 \exp \big( - \tfrac1{50} N \big),\\[4pt]
\end{aligned}
\]
and so \eqref{eq:concentration-l2} holds for any sufficiently large $N$.
\end{proof}


\section{Auxiliary function construction}

The trigonometric polynomial $\varphi$ used in the Riesz product
\eqref{eq:riesz-prod-generators} will be constructed in the following lemma.

\begin{lemma}
\label{lemma:auxiliary}
Given $0 < \eta < 1$ there is a real trigonometric polynomial
$\varphi = \varphi_{\eta}$ such that
\[
\ft{\varphi}(0) = 0, \quad
\|\varphi\|_{\infty} < 1, \quad
\|\varphi\|_{L^2} > \tfrac{9}{10}, \quad
\|\varphi\|_{A_p} \leq C \, \eta^{-1}, \quad
\|\varphi\|_{A_q} \leq C \, \eta,
\]
where $C$ is an absolute constant.
\end{lemma}

We use the common notation $A(\T) = A_1(\T)$ for the Banach space of absolutely
convergent Fourier series. In order to prove Lemma \ref{lemma:auxiliary} we need:

\begin{lemma}
\label{lemma:sum-dilations}
Let $T_1, \dots, T_m \in A_r(\T)$, $1 \leq r < \infty$, $f \in A(\T)$,
$\ft{f}(0) = 0$. Then
\begin{equation}
\label{eq:sum-dilations}
\lim_{\nu \to \infty}
\Big\| \sum_{j=1}^{m} f(\nu^j t) \, T_j \Big\|_{A_r}
= \|f\|_{A_r} \Big( \sum_{j=1}^{m} \|T_j\|_{A_r}^r \Big)^{1/r}.
\end{equation}
\end{lemma}

\begin{proof}
Denote by $P$ and $Q_1, \dots, Q_m$ trigonometric polynomials obtained as partial
sums of the Fourier series of $f$ and $T_1, \dots, T_m$ respectively. We have
\[
\begin{aligned}
&\Big\| \sum_{j=1}^{m} f(\nu^j t) \, T_j - \sum_{j=1}^{m} P(\nu^j t) \, Q_j \Big\|_{A_r}\\
&\qquad \leq \sum_{j=1}^{m} \Big\|f(\nu^j t) (T_j - Q_j)\Big\|_{A_r}
+ \sum_{j=1}^{m} \Big\|(f(\nu^j t) - P(\nu^j t)) Q_j \Big\|_{A_r}\\
&\qquad \leq \|f\|_{A} \sum_{j=1}^{m} \|T_j - Q_j\|_{A_r} + \|f - P\|_{A} \sum_{j=1}^{m}
\| T_j \|_{A_r}.
\end{aligned}
\]
Given $\eps > 0$, we may therefore choose $P$ and $Q_1, \dots, Q_m$ such that
\begin{equation}
\label{eq:sum-dilations-approx}
\Big\| \sum_{j=1}^{m} f(\nu^j t) \, T_j - \sum_{j=1}^{m} P(\nu^j t) \, Q_j \Big\|_{A_r} < \eps.
\end{equation}
If $\nu$ is sufficiently large then each Fourier coefficient of $\sum P(\nu^j t) \, Q_j$
is a product $\ft{P}(k) \ft{Q}_j(l)$, and it is easy to check that this implies
\begin{equation}
\label{eq:sum-dilations-poly}
\Big\| \sum_{j=1}^{m} P(\nu^j t) \, Q_j \Big\|_{A_r}
= \|P\|_{A_r} \Big( \sum_{j=1}^{m} \|Q_j\|_{A_r}^r \Big)^{1/r}.
\end{equation}
By an appropriate choice of $P, Q_1, \dots, Q_m$ the right hand side of
\eqref{eq:sum-dilations-poly} will differ by not more than $\eps$ from the right
hand side of \eqref{eq:sum-dilations}. On the other hand, \eqref{eq:sum-dilations-approx}
allows us to replace $P$ and $Q_j$ in the left hand side of \eqref{eq:sum-dilations-poly}
by $f$ and $T_j$, and the $A_r$ norm of the error will be smaller than $\eps$.
This proves \eqref{eq:sum-dilations}.
\end{proof}

\begin{proof}[Proof of Lemma \ref{lemma:auxiliary}]
Choose and fix a function $\Psi(u)$ on $[0,1]$, infinitely smooth and
vanishing in a neighborhood of the points $0$ and $1$, such that
\[
0 \leq \Psi(u) < 1, \quad  \int_{0}^{1} \Psi(u)^2 \, du  > 9/10.
\]
We shall denote by $C_1, C_2, \dots$ positive constants which depend only on
$\Psi$. We associate with each interval $I = [a, a+h]$ a function $\Psi_I$,
defined on $I$ by
\[
\Psi_I(a + uh) := \Psi(u), \quad u \in [0,1],
\]
and vanishes outside $I$. Let $|I|$ denote the length of $I$, then one may check that
\begin{equation}
\label{eq:interval-dilation-ap}
\|\Psi_I\|_{A_r} \leq C_1 \, |I|^{(r-1)/r}, \;
\text{for any interval $I \subset [0, 2\pi]$ and $1 \leq r < \infty$.}
\end{equation}
Define a $2 \pi$-periodic function $\Phi(t)$ by
\[
\Phi(t) = \Psi_{[0, \pi]}(t) - \Psi_{[\pi, 2 \pi]}(t), \quad
t \in [0, 2\pi].
\]
Partition the segment $[0, 2\pi]$ into $m$ segments $I_1, \dots, I_m$
of equal lengths, and set
\[
\varphi_\nu(t) := \sum_{j=1}^{m} \Phi(\nu^j t) \; \Psi_{I_j}(t), \quad
t \in [0, 2\pi].
\]
Clearly $\|\varphi_\nu\|_\infty < 1$. Now apply Lemma \ref{lemma:sum-dilations},
and use \eqref{eq:interval-dilation-ap} to conclude that
\[
\lim_{\nu \to \infty} \| \varphi_\nu \|_{A_r}
= \|\Phi\|_{A_r} \Big( \sum_{j=1}^{m} \|\Psi_{I_j}\|_{A_r}^r \Big)^{1/r}
\leq C_2 \, m^{(2-r)/r}
\]
for every $1 \leq r < \infty$. Choose $m$ such that
$3^{-1} \eta^{-1} < m^{(2-p)/p} < 3 \eta^{-1}$, then
\[
\lim_{\nu \to \infty} \|\varphi_\nu\|_{A_p} \leq C_3 \, \eta^{-1}, \quad
\lim_{\nu \to \infty} \|\varphi_\nu\|_{A_q} \leq C_3 \, \eta.
\]
We also have
\[
\lim_{\nu \to \infty} \| \varphi_\nu \|_{L^2}
= \|\Phi\|_{L^2} \Big( \sum_{j=1}^{m} \|\Psi_{I_j}\|_{L^2}^2 \Big)^{1/2}
= \int_{0}^{1} \Psi(u)^2 \, du > 9/10,
\]
and, since $\ft{\Phi}(0) = 0$,
\[
\lim_{\nu \to \infty} \ft{\varphi}_\nu(0) = 0.
\]
It is therefore possible to obtain the polynomial $\varphi$ of the lemma by
convolving the function $\varphi_\nu(t) - \ft{\varphi}_\nu(0)$, where $\nu$ is
sufficiently large, with an appropriate Fej\'{e}r kernel.
\end{proof}


\section{Proof of main result}

\subsection{The main lemma}

In this section we prove:

\begin{lemma}
\label{lemma:principal-generators}
Given $\eps > 0$ there is a compact $K$ on the circle, an infinitely differentiable
function $F$ and a real trigonometric polynomial $X$ such that:
\begin{enumerate-math}
\item
$F$ is supported by $K$, \, $\|1 - F\|_{A_q} < \eps$,
\item
$\|X\|_{\infty} \leq 1$, \, $\|X\|_{A_p} < \eps$, \, $X(t) > \frac1{50}$ on $K$.
\end{enumerate-math}
\end{lemma}

\begin{proof}
For $\delta > 0$ to be chosen, use Lemma \ref{lemma:kahane-lemma} to find a measure
$\rho$ supported by the interval $\big(\tfrac{8}{10}, \tfrac{9}{10}\big)$ and
satisfying \eqref{eq:kahane-lemma}. Define
\[
\lambda(t) = \int \lambda_s(t) \, d\rho(s).
\]
In view of \eqref{eq:lambda-trig-expansion-generators}, $\lambda$ is a trigonometric
polynomial whose Fourier expansion is
\[
\lambda(t) = \sum_{k \neq 0} \Big\{ \int s^{l(k)} d\rho(s) 
\prod_{k_j \neq 0} \ft{\varphi}(k_j) \Big\}
\; e^{i(k_1 \nu + k_2 \nu^2 + \cdots + k_N \nu^N) t}.
\]
We therefore have $\ft{\lambda}(0) = 1$, and
\begin{equation}
\label{eq:lambda-sum-q-powers}
\sum_{n \neq 0} |\ft{\lambda}(n)|^q < \delta
\sum_{k \neq 0} \prod_{k_j \neq 0} |\ft{\varphi}(k_j)|^q
= \delta \, (1 + \|\varphi\|_{A_q}^q)^N.
\end{equation}
Consider the compact
\[
K' = \{t \in \T : X(t) \geq \tfrac1{40}\}
\]
where $X$ is given by \eqref{eq:def-x}.
By Lemma \ref{lemma:concentration-l2}, for any $N \geq N(\delta, \rho)$
and for any $\varphi$ satisfying \eqref{eq:assumptions-phi} we can
find $\nu = \nu(N, \varphi) > 2 \deg \varphi$, such that
\[
\Big( \; \int_{\T \setminus K'} \lambda_s^2(t) \, \frac{dt}{2\pi} \;
\Big)^{1/2} < \frac{\delta}{\|\rho\|_M}
\]
holds for all $s \in \supp(\rho)$.
We now choose $\varphi = \varphi_\eta$ to be the trigonometric polynomial
given by Lemma \ref{lemma:auxiliary} with
\begin{equation}
\label{eq:def-eta}
\eta = \eps^{-1} N^{-1/q},
\end{equation}
where $N$ is taken large enough such that $0 < \eta < 1$.
Then
\[
(1 + \|\varphi\|_{A_q}^q)^N
\leq \exp (N \|\varphi\|_{A_q}^q)
\leq \exp (C^q \eta^q  N)
\leq \exp (C^q  \eps^{-q}),
\]
so from \eqref{eq:lambda-sum-q-powers} it follows that
\[
\|1 - \lambda\|_{A_q} \leq C(p, \eps) \; \delta^{1/q}.
\]
Denote $H = \lambda \cdot \1_{K'}$, then
\[
\|\lambda - H\|_{A_q} \leq \|\lambda - H\|_{L^2(\T)}
= \|\lambda\|_{L^2(\T \setminus K')} \leq
\int_{\T \setminus K'} \|\lambda_s(t)\|_{L^2(\T \setminus K')} \; d|\rho|
< \delta.
\]
It is therefore clear that if $\delta = \delta(p, \eps)$ is chosen sufficiently
small, we get $\|1 - H\|_{A_q} < \eps$.
We then define the function $F$ as the convolution of $H$ with a smooth
non-negative function with integral $= 1$, which is supported by a sufficiently
small neighbourhood of $0$ to ensure that $X(t) > \tfrac1{50}$ on $K := \supp(F)$.

Finally, we show that the trigonometric polynomial $X$ satisfied the required
conditions. Indeed, $\|X\|_{\infty} \leq 1$, and using the fact that
$\nu > 2 \deg \varphi$ we also have
\[
\|X\|_{A_p} = N^{1/p-1} \|\varphi\|_{A_p}
\leq C \, N^{1/p - 1} \, \eta^{-1}
= C \, \eps.
\]
The lemma is thus proved.
\end{proof}

\subsection{Successive approximations}
For each $n = 0,1,2,\dots$ we shall construct an infinitely differentiable
function $f_n: \T \to \C$ and two trigonometric polynomials $g_n$ and
$P_n$, such that the following properties hold. For all $n \geq 0$,
\begin{enumerate-math}
\renewcommand{\labelenumi}{$\text{\theenumi}_n$}
\item
\label{item:gn-small}
$|g_n(t)| \leq (\frac{99}{100})^n$ on $\supp(f_n)$,
\item
\label{item:chi-almost-inverse}
$\|1 - P_n \cdot g_n\|_{A_p} < 2^{-n-1}$,
\end{enumerate-math}
and for all $n \geq 1$,
\begin{enumerate-math}
\renewcommand{\labelenumi}{$\text{\theenumi}_n$}
\setcounter{enumi}{2}
\item
\label{item:gn-approx-c}
$\|g_{n-1} - g_n\|_{\infty} \leq (\frac{99}{100})^{n-1}$,
\item
\label{item:gn-approx-ap}
$\|g_{n-1} - g_n\|_{A_p} < 2^{-n-1}(1 + \|P_0\|_A + \|P_1\|_A + \dots + \|P_{n-1}\|_A)^{-1}$,
\item
\label{item:fn-approx}
$\|f_{n-1} - f_n\|_{A_q} < 2^{-n-1}$,
\item
\label{item:fn-support}
$\supp(f_{n}) \subset \supp(f_{n-1})$.
\end{enumerate-math}

The construction is done by induction. We define $f_0 = g_0 = P_0 = 1$,
and note that the properties $\text{\ref{item:gn-small}}_{0}$ and
$\text{\ref{item:chi-almost-inverse}}_{0}$ are satisfied.

Suppose now that $f_k$, $g_k$ and $P_k$ have been constructed for all
$0 \leq k \leq n$. Then $f_{n+1}$, $g_{n+1}$ and $P_{n+1}$ can be
constructed as follows. The property $\text{\ref{item:gn-small}}_{n}$ ensures that one
can find a trigonometric polynomial $h$ such that $\|h\|_{\infty} \leq
(\frac{99}{100})^n$ and
\begin{equation}
\label{eq:gn-h-approx-supp-fn}
|g_{n}(t) - h(t)| < \tfrac1{100} \cdot (\tfrac{99}{100})^n,
\quad \text{for every $t \in \supp(f_n)$.}
\end{equation}
For $\eps > 0$ to be chosen, let $K$, $F$ and $X$ be given by Lemma
\ref{lemma:principal-generators}, and set
\begin{equation}
\label{eq:fn-gn-definition}
f_{n+1} := f_n \cdot F, \quad
g_{n+1} := g_n - h \cdot X.
\end{equation}
Indeed $f_{n+1}$ is an infinitely differentiable, and $g_{n+1}$ is a trigonometric
polynomial. Since we have
\[
\|g_{n} - g_{n+1}\|_{A_p} = \| h \cdot X \|_{A_p} \leq 
\| h \|_{A} \, \| X \|_{A_p} \leq \eps \| h \|_{A},
\]
and
\[
\|f_{n} - f_{n+1}\|_{A_q} = \|f_n \cdot (1 - F)\|_{A_q} \leq
\|f_n\|_A \, \|1 - F\|_{A_q} \leq \eps \|f_n\|_A,
\]
for a sufficiently small $\eps$ the properties
$\text{\ref{item:gn-approx-ap}}_{n+1}$ and $\text{\ref{item:fn-approx}}_{n+1}$
will be satisfied. Also
\[
\|g_{n} - g_{n+1}\|_{\infty} = \| h \cdot X \|_{\infty} \leq 
\| h \|_{\infty} \, \| X \|_{\infty} \leq (\tfrac{99}{100})^n,
\]
which gives $\text{\ref{item:gn-approx-c}}_{n+1}$. Next,
note that $\supp(f_{n+1}) \subset \supp(f_n) \cap K$ and so property
$\text{\ref{item:fn-support}}_{n+1}$ is satisfied. Also, the fact that
$|1 - X(t)| < \frac{49}{50}$ on $K$, together with \eqref{eq:gn-h-approx-supp-fn},
allows to conclude that for every $t \in \supp(f_{n+1})$
\[
\begin{aligned}
|g_{n+1}(t)| &= |g_n(t) - h(t) X(t)|
= |g_n(t) - h(t) + (1 - X(t)) h(t)|\\
&\leq |g_n(t) - h(t)| + |1 - X(t)| \cdot |h(t)|\\
&\leq \tfrac1{100} \cdot (\tfrac{99}{100})^n + \tfrac{49}{50} \cdot (\tfrac{99}{100})^n
= (\tfrac{99}{100})^{n+1},
\end{aligned}
\]
hence $\text{\ref{item:gn-small}}_{n+1}$ holds as well.
Finally, note that $\text{\ref{item:gn-approx-ap}}_{1}$,
$\text{\ref{item:gn-approx-ap}}_{2}$,
\dots, $\text{\ref{item:gn-approx-ap}}_{n+1}$ imply that
\begin{equation}
\label{eq:gn-g0-dist-ap}
\|1 - g_{n+1}\|_{A_p} \leq \sum_{k=1}^{n+1} \|g_{k-1} - g_{k}\|_{A_p} <
\sum_{k=1}^{n+1} 2^{-k-1} < 2^{-1},
\end{equation}
hence the trigonometric polynomial $g_{n+1}$ is non-zero. One can therefore find
a trigonometric polynomial $P_{n+1}$ satisfying $\text{\ref{item:chi-almost-inverse}}_{n+1}$ 
(see \cite{kahane-salem}, Chapter VIII, \textsection 6).

The constructed $f_{n+1}$, $g_{n+1}$ and $P_{n+1}$ thus satisfy all
the required properties.

\subsection{Conclusion of the proof}

%
The properties $\text{\ref{item:fn-approx}}_{n}$ and $\text{\ref{item:fn-support}}_{n}$,
together with $f_0 = 1$, imply that the sequence $f_n$ converges in $A_q$ to a non-zero
distribution $S$ such that $\supp(S) \subset \bigcap_{n=1}^{\infty} \supp(f_n)$. The
properties $\text{\ref{item:gn-approx-c}}_{n}$ and $\text{\ref{item:gn-approx-ap}}_{n}$
imply that the sequence $g_n$ converges both uniformly and in $A_p(\T)$ to a function
$g \in C(\T) \cap A_p(\T)$. From $\text{\ref{item:gn-small}}_{n}$ it is seen that $g(t) = 0$
for every $t \in \supp(S)$. In addition, using $\text{\ref{item:gn-approx-ap}}_{n}$ we
have
\[
\|P_n \cdot (g_n - g)\|_{A_p}
\leq \|P_n\|_A \sum_{k = n+1}^{\infty} \|g_{k-1} - g_{k}\|_{A_p}
\leq \sum_{k = n+1}^{\infty} 2^{-k-1} = 2^{-n-1},
\]
and therefore
\[
\|1 - P_n \cdot g\|_{A_p} \leq \|1 - P_n \cdot g_n\|_{A_p} + \|P_n \cdot (g_n - g)\|_{A_p}
< 2^{-n-1} + 2^{-n-1} = 2^{-n}.
\]
That is, the function $1$ can be approximated in $A_p(\T)$ by functions of the
form $P \cdot g$, where $P$ is a trigonometric polynomial. This
easily implies that any trigonometric polynomial can be approximated by these
functions, which shows that the set \eqref{eq:g-ideal} is dense in $A_p(\T)$.
The function $g$ therefore satisfies the two conditions given in Theorem
\ref{thm:strong-piatetski}, so the proof is complete.


\section{The $L^p(\R)$ version}

In this section we prove Theorem \ref{thm:impossible-real}.

\subsection{}
Being given a Schwartz function $\gamma(x)$ on $\R$, consider the linear
operator $T$ which maps each trigonometric polynomial
\[
f(t) = \sum_{|n| \leq N} \ft{f}(n) e^{int}
\]
to the function $Tf : \R \to \C$ defined by
\begin{equation}
\label{eq:def-t-gamma}
(T f)(x) = \sum_{|n| \leq N} \ft{f}(n) \, \gamma(x + n), \quad x \in \R.
\end{equation}
Clearly $Tf$ is a Schwartz function, for every trigonometric polynomial $f$.

\begin{lemma}
\label{lemma:t-operator-bounded}
$\|Tf\|_{L^p(\R)} \leq M \|f\|_{A_p(\T)}$, the constant $M$
not depending on $f$.
\end{lemma}

\begin{proof}
We must show that for any $N$ and scalars $\{c(n)\}$ one has
\[
\Big\| \sum_{|n| \leq N} c(n) \, \gamma(x + n) \Big\|_{L^p(\R)}
\leq M \, \big\|\{c_n\}\big\|_{\ell_p}
\]
for some $M = M(\gamma)$. 
An application of H\"{o}lders's inequality yields
\[
\Big| \sum_{|n| \leq N} c(n) \, \gamma(x + n) \Big|
\leq
\Big( \sum_{|n| \leq N} |c(n)|^p \, |\gamma(x + n)| \Big)^{1/p}
\Big( \sum_{|n| \leq N} |\gamma(x + n)| \Big)^{1/q}.
\]
Choose $M$ large enough so that
\[
\sup_{x \in \R} \; \sum_{n \in \Z} |\gamma(x + n)| \; \leq \; M.
\]
Then
\[
\Big| \sum_{|n| \leq N} c(n) \, \gamma(x + n) \Big|^p \leq 
M^{p-1} \sum_{|n| \leq N} |c(n)|^p \, |\gamma(x + n)|.
\]
If $M$ is sufficiently large then also $\int_{\R} |\gamma(x)| \, dx \leq M$,
thus integration yields
\[
\int_{\R} \Big| \sum_{|n| \leq N} c(n) \, \gamma(x + n) \Big|^p \; dx
\leq M^p \sum_{|n| \leq N} |c(n)|^p.
\]
This proves the lemma.
\end{proof}

Lemma \ref{lemma:t-operator-bounded} allows to extend the linear operator $T$
by continuity to a bounded operator $T: A_p(\T) \to L^p(\R)$. It follows from
\eqref{eq:def-t-gamma} that $T$ has the properties
\begin{equation}
\label{eq:t-prop-mult}
T(e^{int} f) (x) = (Tf)(x + n),
\end{equation}
and
\begin{equation}
\label{eq:t-prop-ft}
\ft{{(T f)}} (t) = \ft{\gamma}(t) f(t), \; t \in \R,
\end{equation}
where in \eqref{eq:t-prop-ft} the function $f$ is considered as a
$2\pi$-periodic function, and the equality is understood in the sense
of distributions, or almost everywhere.

\subsection{}
Choose a Schwartz function $\gamma$ whose Fourier transform $\ft{\gamma}$ has no zeros,
and let $g \in C(\T) \cap A_p(\T)$ be the function given by Theorem \ref{thm:strong-piatetski}.
Define a function $G \in L^p(\R)$ by $G := T g$. For each trigonometric polynomial $P$,
\[
\|\gamma - T(P \cdot g)\|_{L^p(\R)} = \|T(1 - P \cdot g)\|_{L^p(\R)}
\leq \|T\| \cdot \|1 - P \cdot g\|_{A_p(\T)}.
\]
The set $\{P(t) g(t)\}$ is dense in $A_p(\T)$, so in particular the constant function
which is equal to $1$ identically, can be approximated by products $P(t)g(t)$. It follows
that $\gamma$ can be approximated in $L^p(\R)$ by functions of the form $T(P \cdot g)$.
Observe that $T(P \cdot g)$ is a finite linear combination of translates of $G$ (in fact,
integer translates), according to \eqref{eq:t-prop-mult}. The function $\gamma$ therefore
belongs to the closed linear subspace of $L^p(\R)$ generated by the translates of $G$. But
$\gamma$ is a Schwartz function with a non-vanishing Fourier transform, so it is a generator
in $L^p(\R)$ (see \cite{beurling}). Hence $G$ is a generator as well.

On the other hand, choose an infinitely differentiable function $f: \T \to \C$ such that
$Z_f = Z_g$, and define $F := Tf \in L^p(\R)$. By \eqref{eq:t-prop-ft} we have $\ft{F}(t)
= \ft{\gamma}(t) f(t)$ and $\ft{G}(t) = \ft{\gamma}(t) g(t)$, and since $\ft{\gamma}$ has
no zeros it follows that $\ft{F}$ and $\ft{G}$ are continuous functions having the same
zero set. We also note that $F$ is a Schwartz function. Let us show that $F$ is not a
generator. Indeed, the set $Z_f \subset \T$ supports a non-zero distribution $S$ belonging
to $A_q(\T)$. Since $Z_f$ is not the whole circle, we may identify $S$ with a distribution
on $\R$, supported by an interval of length $< 2\pi$, whose Fourier transform belongs to
$L^q(\R)$ (see \cite{boas}, Corollary 10.6.6, p. 197). But the zero set of $\ft{F}$ (i.e.
the $2\pi$-periodization of $Z_f$) contains the supports of $S$, hence $F$ cannot be a
generator (see \cite{pollard}, Theorem B).

This completes the proof of Theorem \ref{thm:impossible-real}.

\subsection{}
We conclude this chapter by clarifying our remark made earlier, that the functions $F$ and
$G$ in Theorem \ref{thm:impossible-real} could in fact be chosen as the restrictions to the
real line of two entire functions of order $1$.

Indeed, we can choose our function $\gamma$ such that its Fourier transform $\ft{\gamma}$
decays arbitrarily fast. In particular, we may assume that for every positive number $r$,
\begin{equation}
\label{eq:gamma-ft-decay}
|\ft{\gamma}(t)| = O(e^{- |t|^{r}}), \quad |t| \to \infty.
\end{equation}
In this case the integral
\begin{equation}
\label{f-extension-entire}
F(z) = \frac1{2\pi} \int_{\R} \ft{F}(t) \, e^{izt} \, dt, \quad z = x + iy,
\end{equation}
will converge absolutely and uniformly in every strip $|y| \leq y_0$, hence a standard
argument shows that $F(z)$ is an analytic function in the whole complex plane.
Let us prove that it has order $1$. It will be enough to show that
$|F(z)| = O(e^{|z|^{\alpha}})$ for every $\alpha > 1$. Obviously there is the estimate
\[
|F(z)| \leq C \int_{\R} |\ft{\gamma}(t)| \, e^{|y| \cdot |t|} \, dt, \quad C := \|f\|_\infty.
\]
Let now $\beta$ denote the exponent conjugate to $\alpha$ satisying
$1/\alpha + 1/\beta = 1$. Then by Young's inequality
$|y| \cdot |t| \leq \tfrac1{\alpha} |y|^{\alpha} + \tfrac1{\beta} |t|^{\beta}$,
and consequently
\[
|F(z)| \leq C \exp \big( \tfrac1{\alpha} |y|^{\alpha} \big)
\int_{\R} |\ft{\gamma}(t)| \exp \big( \tfrac1{\beta} |t|^{\beta} \big) dt.
\]
The condition \eqref{eq:gamma-ft-decay} now implies the convergence of the integral
on the right hand side, so we get $|F(z)| = O(\exp(\tfrac1{\alpha} |y|^{\alpha}))$.
This shows that $F$ is an entire function of order $1$, and by similar arguments the
same conclusion is true for the function $G$.


\chapter{Related Problems}

In this chapter we study further aspects of Piatetski-Shapiro's phenomenon in $\ell_q$
$(q>2)$, and obtain several strengthenings of Theorem \ref{thm:piatetski-theorem-lq}.
First we determine how small could be a compact satisfying Piatetski-Shapiro's phenomenon,
and in particular, how small can be the Hausdorff dimension of such a compact. Also
the relation with a certain Fourier-analytic dimension will be discussed. Another question
under consideration is whether Piatetski-Shapiro's phenomenon is ``typical'' or ``rare''.
Lastly, we study the relation with sets of interpolation for the class $(A_p \cap C)(\T)$,
namely $p$-Helson sets.


\section{Shrinking method, Hausdorff dimension}
\label{sec:shrinking-hausdorff}

\subsection{Introduction}
In Section \ref{subsec:extensions-of-piatetski-result} we referred to the theorem
of Kaufman: any set of multiplicity contains a Helson set of multiplicity. In this
section we are interested in the following corollary of Kaufman's theorem:

\emph{Let $K$ be a compact which supports a non-zero distribution $S$ with
Fourier coefficients tending to zero. Then there is a compact $K_1 \subset K$,
which also supports a non-zero distribution $S_1$ with Fourier coefficients
tending to zero, but which does not support such a measure.}

In other words, any compact of multiplicty can be ``shrinked'' to a compact of 
Piatetski-Shapiro type. It follows, for example, that such a compact can be
``as small as desired'' in the sense of Hausdorff dimension. Indeed, it was
proved by Ivashev-Musatov \cite{ivashev-musatov-2} that
for any continuous non-decreasing function $h(t)$ defined for $t \geq 0$,
$h(0) = 0$, there exists a set of multiplicity (even in the restricted sense)
with Hausdorff $h$-measure zero. It then follows that there is a compact $K$
which has this property and which satisfies Piatetski-Shapiro's theorem.

It is interesting to remark that the Hausdorff dimension of the original
Piatet\-ski-Shapiro compact $K_\gamma$ $(0 < \gamma < 1/2)$ (see Section
\ref{subsec:about-piatetski-result}) was computed by Besicovitch
\cite{besicovitch}. It is the number $\alpha \in (0,1)$ determined by the equation
\[
2^\alpha = \frac1{\gamma^\gamma (1 - \gamma)^{1 - \gamma}}.
\]
Observe that by a convenient choice of $\gamma$ the dimension $\alpha$ can be
made arbitrarily small, but it remains positive.

%
%

\subsection{Result}
Here we obtain a strengthening of Theorem \ref{thm:piatetski-theorem-lq}
which is analogous to the above mentioned ``shrinking theorem''.
We prove the following result.

\begin{theorem}
\label{thm:kaufman-theorem-lq}
Suppose a compact $K \subset \T$ supports a non-zero distribution
$S$ such that $\ft{S} \in \ell_q$ $(q > 2)$. Then there is a compact
$K_1 \subset K$, which also supports a non-zero distribution $S_1$ with
$\ft{S}_1 \in \ell_q$, but which does not support such a measure.
\end{theorem}


In particular, this allows us to determine the smallest possible Hausdorff
dimension of a compact satisfying Theorem \ref{thm:piatetski-theorem-lq}.
It is known that a compact $K$ which supports a distribution $S$ with $\ft{S}
\in \ell_q$, has Hausdorff dimension $\geq 2/q$ (see \cite{kahane-salem},
Chapter VIII, \textsection 4). On the other hand, there are examples of
compacts $K$ of Hausdorff dimension $= 2 / q$ which support such a distribution
(and even a positive measure); this may be deduced from a result of K\"{o}rner
\cite[Theorem 1.2]{korner:ivashev-musatov} or, alternatively, in
\cite[Theorem 5]{rosenblatt-shuman} it is explained how to deduce this from
results of Kahane \cite{kahane:random}. By applying Theorem
\ref{thm:kaufman-theorem-lq} to such a compact we obtain:

\begin{corollary}
For any $q > 2$ there exists a compact $K$ on the circle, which
satisfies the result of Theorem \ref{thm:piatetski-theorem-lq}, and
which has Hausdorff dimension $2 / q$.
\end{corollary}

\subsection{Shrinking lemma}
We will use the shrinking method of Kaufman, which corresponds to the multiplication
of a distribution $S$ by a function $f$ in order to ``shrink'' its support, see
\cite[Chapter VII]{kechris-louveau} and \cite[Theorem 4.6.2]{graham-mcgehee} for
more details. The proof of Theorem \ref{thm:kaufman-theorem-lq} will depend on the
following replacement for Lemma \ref{lemma:principal-lq}.

\begin{lemma}
\label{lemma:principal-lq-kaufman}
Let $K$ be a compact on the circle, and suppose that $S \in A_q$ is a
non-zero distribution supported by $K$. Being given $\eps > 0$ and a positive
integer $\nu$, one can find a compact $K_1 \subset K$ with the following
properties:
\begin{enumerate-math}
\item
$K_1$ supports a distribution $S_1 \in A_q$,
$\|S_1 - S\|_{A_q} < \eps$.
\item
There is a real trigonometric polynomial
\[
X(t) = \sum_{|n| \geq \nu} \ft{X}(n) e^{int},
\quad \|\ft{X}\|_p \leq 1, \quad
\text{$\tfrac1{100} \leq X(t) \leq 100$ on $K_1$}.
\]
\end{enumerate-math}
\end{lemma}


In order to prove Lemma \ref{lemma:principal-lq-kaufman} we shall need the following
lemma (also compare with Lemma \ref{lemma:sum-dilations}).

\begin{lemma}
\label{lemma:multiplication-lq-norm}
Let $f$ be a function in $A$, and $S$ be a distribution in $A_q$.
For any integer $m$ we denote $f_m(t) = f(mt)$, and consider the
product $S f_m$. Then
\[
\lim_{m \to \infty} \|S f_m\|_{A_q} = \|f\|_{A_q} \|S\|_{A_q}.
\]
\end{lemma}

\begin{proof}
Given two trigonometric polynomials $P$ and $Q$, consider the product
$Q P_m$, where we denote $P_m(t) = P(mt)$. It is easy to check that if
$m$ is sufficiently large then $\|Q P_m\|_{A_q} = \|Q\|_{A_q} \|P\|_{A_q}$,
since each non-zero Fourier coefficient of $Q P_m$ will be a product $\ft{P}(j)
\ft{Q}(k)$.

Given $\eps > 0$, we choose $P$ and $Q$ as partial sums of the Fourier series
of $f$ and $S$, respectively. By an appropriate choice we may have the
difference between $\|Q\|_{A_q} \|P\|_{A_q}$ and $\|S\|_{A_q} \|f\|_{A_q}$
be smaller than $\eps$ by modulus. In addition, by a standard argument we have
\[
\|S f_m - Q P_m\|_{A_q} \leq
\|S - Q\|_{A_q} \|f\|_{A} + \|Q\|_{A_q} \|f - P\|_{A} \, ,
\]
so we may also satisfy
\[
\|S f_m - Q P_m\|_{A_q} < \eps
\]
for every $m$. It follows that for sufficiently large $m$, the difference
between $\|S f_m\|_{A_q}$ and $\|f\|_{A_q} \|S\|_{A_q}$ is smaller that
$2 \eps$ by modulus.
\end{proof}

\begin{proof}[Proof of Lemma \ref{lemma:principal-lq-kaufman}]
Using Lemma \ref{lemma:principal-lq} one can find a compact $E \subset \T$,
a $C^\infty$ function $f: \T \to \C$ and a real trigonometric polynomial $X$
such that
\begin{enumerate-math}
\item
$f$ is supported by $E$, $\|f - 1\|_{A_q} < \eps / \|S\|_{A_q}$.
\item
$X(t) = \sum_{|n| \geq \nu} \ft{X}(n) e^{int}$, $\|\ft{X}\|_p \leq 1$
and $\tfrac1{100} \leq X(t) \leq 100$ on $E$.
\end{enumerate-math}
Define $f_m(t) = f(mt)$, $X_m(t) = X(mt)$ and let $E_m$ denote
the compact defined by $\1_{E_m}(t) = \1_{E}(mt)$. Then for every
positive integer $m$ the conditions (i) and (ii) are also satisfied
by $E_m$, $f_m$ and $X_m$ instead of $E$, $f$ and $X$. Let $S_m$
denote the product $S f_m$, then due to Lemma
\ref{lemma:multiplication-lq-norm}
\[
\lim_{m \to \infty} \|S_m - S\|_{A_q} = \lim_{m \to \infty}
\|S(f_m - 1)\|_{A_q} = \|S\|_{A_q} \|f - 1\|_{A_q} < \eps.
\]
Hence for sufficiently large $m$, the compact $K \cap E_m$ satisfies
the conditions (i) and (ii) of Lemma \ref{lemma:principal-lq-kaufman}.
\end{proof}

\subsection{Conclusion of the proof}
We now deduce Theorem \ref{thm:kaufman-theorem-lq} from Lemma
\ref{lemma:principal-lq-kaufman}. Let $K$ be a compact on the circle,
and let $S$ be a non-zero distribution in $A_q$ supported by $K$. Choose
numbers $\eps_j > 0$, $\sum_{j=1}^{\infty} \eps_j < \|S\|_{A_q}$
and positive integers $\nu_j$ such that
 and $\nu_j \to \infty$.
$(j \to \infty)$. Using Lemma \ref{lemma:principal-lq-kaufman} one may
construct by induction compacts $K_j$ on the circle ($K_0 = K$),
distributions $S_j$ in $A_q$ ($S_0 = S$) and real trigonometric
polynomials $X_j$ such that, for every $j \geq 1$
\begin{enumerate-math}
\item
$K_j \subset K_{j-1}$.
\item
$S_j$ is supported by $K_j$.
\item
$\|S_j - S_{j-1}\|_{A_q} < \eps_j$.
\item
$X_j(t) = \sum_{|n| \geq \nu_j} \ft{X}_j(n) e^{int}$,
$\|\ft{X}_j\|_p \leq 1$ and $\frac1{100} \leq X_j(t) \leq 100$ on $K_j$.
\end{enumerate-math}

It follows that the sequence $S_j$ conveges in the $A_q$ norm to a non-zero
distribution $S_\infty$ supported by $K_\infty = \bigcap_{j=1}^{\infty} K_j$.
By Lemma \ref{lemma:nomeasure-lq}, $K_\infty$ does not support a non-zero
measure $\mu$ in $A_q$. We have thus obtained a compact $K_\infty$ as required,
so the theorem is proved.


\section{Fourier-type dimension}

\subsection{}
The \emph{Fourier dimension} of a compact $K \subset \T$ is usually defined
as the supremum of the numbers $0 \leq \beta \leq 1$ for which there is a positive
(non-zero) measure $\mu$, supported by $K$, such that $|\ft{\mu}(n)| = 
O(|n|^{-\beta/2})$. A well-known property of the Fourier dimension is that
it can never exceed the Hausdorff dimension of $K$ (see \cite{kahane-salem},
Chapter VIII). The Fourier dimension need not coincide with the Hausdorff
dimension in general. This is not surprising since the Hausdorff dimension is
related to metrical properties of the set, while the Fourier dimension is
closely related to its arithmetical properties.

In this section we consider a quantity similar to the Fourier dimension. We
define the number $q_0 = q_0(K)$ to be the infimum of all $q > 2$ for which
there is a positive measure $\mu$, supported by $K$, such that $\ft{\mu}
\in \ell_q$. The number $q_0$, like the Fourier dimension, could be used to
estimate the Hausdorff dimension from below: the set $K$ has Hausdorff
dimension $\geq 2/q_0$ (see Section \ref{sec:shrinking-hausdorff}).

Rosenblatt and Shuman in \cite{rosenblatt-shuman} used results of Beurling
\cite{beurling}, Salem \cite{salem} and Kahane \cite{kahane:random}
to construct compact sets $K$ with $q_0$ being any pre\-assigned number,
$q_0 > 2$. They constructed such examples in which $K$ supports a positive
measure with Fourier coefficients belonging to $\ell_{q_0}$, and also other
examples in which $K$ does not even support such a distribution.

Here we construct an example which exhibits a different ``threshold behavior''.
The following strengthening of Theorem \ref{thm:piatetski-theorem-lq} is true:
there is a compact $K$ which does not support any measure $\mu$ with $\ft{\mu}
\in \ell_{q_0}$, $q_0 = q_0(K)$, but which does support such a \emph{distribution}.
In fact, we prove a bit more than that:

\begin{theorem}
\label{thm:piatetski-theorem-lq-measures}
For any $q_0 > 2$ there exists a compact $K$ such that
\begin{enumerate-math}
\item
$K$ supports a positive measure $\mu$ such that $\ft{\mu} \in \bigcap_{q > q_0} \ell_q$.
\item
$K$ does not support any measure with Fourier coefficients belonging to $\ell_{q_0}$.
\item
$K$ supports a distribution $S$ with Fourier coefficients belonging to $\ell_{q_0}$.
\end{enumerate-math}
\end{theorem}

\subsection{}
In order to prove Theorem \ref{thm:piatetski-theorem-lq-measures} we add
a new ingredient to Lemma \ref{lemma:principal-lq}.

\begin{lemma}
\label{lemma:principal-lq-measures}
Let $q > 2$ be given. For any $\eps > 0$ and any positive integer $\nu$
one can find a compact $K$ on the circle, two $C^\infty$ functions
$f, g: \T \to \C$ and a real trigonometric polynomial $X$ such that
\begin{enumerate-math}
\item
$f$ is supported by $K$.
\item
$f(t) = 1 + \sum_{n \neq 0} \ft{f}(n)e^{int}$, $\|f - 1\|_{A_q} < \eps$.
\item
$g$ is supported by $K$.
\item
$g(t) \geq 0$, $t \in \T$.
\item
$g(t) = 1 + \sum_{n \neq 0} \ft{g}(n)e^{int}$, $\|g - 1\|_{A_{q + \eps}} < \eps$.
\item
$X(t) = \sum_{|n| \geq \nu} \ft{X}(n) e^{int}$,
$\|\ft{X}\|_p \leq 1$, $\tfrac1{100} \leq X(t) \leq 100$ on $K$.
\end{enumerate-math}
\end{lemma}

\begin{proof}
We choose $f$, $K$ and $X$ according to Lemma \ref{lemma:principal-lq}. Using the
notations of this lemma, we now explain how to choose the function $g$.
In view of 
\ref{eq:lambda-trig-expansion-lq}, for any $\lowers < s < \uppers$ we have
\[
\sum_{n \neq 0} \big|\ft{\lambda}_s(n)\big|^{q + \eps}
< \sum_{0 \neq \bar{\tau} \in \{-1,0,1\}^N}
\big( \tfrac1{2} \, N^{-(q + \eps)/q} \big)^{\sum |\tau_j|}
= \big(1 + N^{- 1 - \eps/q} \big)^N - 1,
\]
from which we deduce the estimate
\begin{equation}
\label{eq:lambda-norm-lq-epsilon}
\sum_{n \neq 0} \big|\ft{\lambda}_s(n)\big|^{q + \eps}
\leq \exp \big( N^{- \eps / q} \big) - 1 \, .
\end{equation}
Let $g_1(t) = \lambda_s \cdot \1_{K'}$ for some arbitrary choice of
$\lowers < s < \uppers$. Then $g_1$ is a positive function supported by
$K'$, and
\[
\|\lambda_s - g_1\|_{A_{q + \eps}} \leq \|\lambda_s - g_1\|_{L^2(\T)}
= \|\lambda_s\|_{L^2(\T \setminus K')} < \delta.
\]
The right side of \eqref{eq:lambda-norm-lq-epsilon} is also smaller
than $\delta$, if $N$ is sufficiently large. So, as before, we can define
$g_2(t) = g_1(t) / \ft{g}_1(0)$ and take $g$ to be the convolution of $g_2$
with an appropriate smooth kernel.
\end{proof}

\subsection{}
We now conclude the proof with a similar procedure as before.

\begin{proof}[Proof of Theorem \ref{thm:piatetski-theorem-lq-measures}]
As in the proof of Theorem \ref{thm:piatetski-theorem-lq}, we choose by induction
a sequence $\{\eps_j\}$ and let $f_j$, $g_j$, $X_j$ and $K_j$ be given by Lemma
\ref{lemma:principal-lq-measures} (with $q = q_0$). If the $\{\eps_j\}$ are
tending to zero fast enough then the product $\prod_{j=1}^{\infty} f_j$ will
converge in $A_{q_0}$ to a non-zero distribution $S$, which is supported by
the compact $K = \bigcap_{j=1}^{\infty} K_j$. Lemma \ref{lemma:nomeasure-lq}
implies that $K$ supports no non-zero measure in $A_{q_0}$.

We define the measure $\mu$ as the product $\prod_{j=1}^{\infty} g_j$. Denote
$q_j = q_0 + \eps_j$, $G_0 = 1$ and $G_j = g_1 \cdots g_j$. Given any $q > q_0$,
we have $q_j < q$ for all sufficiently large $j$, which implies
\[
\|G_{j+1} - G_j\|_{A_{q}} \leq \|G_{j+1} - G_j\|_{A_{q_{j+1}}} 
\leq \|G_j\|_{A} \cdot \|g_{j+1} - 1\|_{A_{q_{j+1}}}.
\]
Hence if we choose
\[
\eps_1 < 2^{-2} \qquad \textrm{and} \qquad
\|g_1 \cdot g_2 \cdot \ldots \cdot g_j\|_{A} \; \eps_{j+1} <
2^{-2-j} \quad (j=1,2,\dots)
\]
then the product $\prod_{j=1}^{\infty} g_j$ will converge in $A_{q}$, for all $q > q_0$,
to a distribution $\mu$. For each $j$, the function $g_j$ is positive and supported by
$K_j$, so $\mu$ is a positive measure supported by $K$. Finally, we have
\[
|\ft{G}_{j+1}(0) - \ft{G}_j(0)| \leq \|G_{j+1} - G_j\|_{A_{q_{j+1}}} < 2^{-2-j},
\]
and therefore $|\ft{\mu}(0) - 1| < \sum_{0}^{\infty} 2^{-2-j} < 1$, so $\mu$ is
non-zero.
\end{proof}

\section{Baire category}

In this section we consider the question: is the phenomenon of Theorem
\ref{thm:piatetski-theorem-lq} ``typical'' or ``rare''?

One possible approach to such type of questions is via probability theory.
In this context, the introduction of a probability measure space allows one
to consider an almost sure event as ``typical'', and the complement of such
an event as ``rare''. Probabilistic methods in the theory of thin sets
were used by Kahane and Salem, see \cite{kahane-salem} and \cite{kahane:random}.

Another approach is via Baire category. In this case, the introduction of a 
complete metric space is required. Rare phenomena then correspond to sets of
first category, also known as meager sets, while typical phenomena correspond
to their complements, namely the residual sets (we recall the precise definitions
below). In connection with thin sets, Baire category arguments were used by
Kaufman \cite{kaufman:functional}, Kahane \cite{kahane:rearrangements} and
K\"{o}rner \cite{korner:curve}. For more details see \cite{kahane:baire}.

Inspired by K\"{o}rner's paper \cite{korner:curve} we use the Baire category
approach to show that the phenomenon of Theorem \ref{thm:piatetski-theorem-lq}
is ``typical'' in an appropriate sense.

\subsection{Preliminaries}

We briefly recall the notions of Baire category theory. Let $\X$ be a complete
metric space. A subset of $\X$ is called \emph{nowhere dense} if its closure has
no interior points. Equivalently, a set is nowhere dense if its complement contains
an open, dense set. A subset of $\X$ is called a \emph{set of first category},
or a \emph{meager} set, if it is contained in the union of countably many nowhere
dense sets. The complement of a set of first category is called a \emph{residual}
set. Equivalently, a set is residual if it contains the intersection of countably
many open dense sets. The Baire category theorem reads as follows.

\begin{theorem*}[Baire]
In a complete metric space, any residual set is dense.
\end{theorem*}

The meager sets in the theory of Baire category play the same role as the sets of
measure zero in probability theory. If the set of elements $x \in \X$ satisfying
a certain property $P$ is residual, we will say that $P$ is ``typical'' or
``quasi-sure''. It is also common to use statements like ``$P$ holds quasi-surely''
or ``quasi-all elements $x \in \X$ satisfy $P$''.

\subsection{Result}

We denote by $\K(\T)$ the collection of all non-empty, compact subsets
of $\T$. We define the \emph{$\eps$-neighborhood} of a compact $K \in \K(\T)$
to be the set
\[
K_\eps = \bigcup_{t \in K} (t - \eps, t + \eps).
\]
In other words, $K_\eps$ is the open set consisting of all points whose distance
from $K$ is smaller than $\eps$. Given two compacts $K, K' \in \K(\T)$, we define
the \emph{Hausdorff distance} between $K$ and $K'$ to be
\[
\delta(K, K') = \inf \{\eps > 0 : \text{$K \subset K'_\eps$ and $K' \subset K_\eps$}\}.
\]
It is well-known that the space $\K(\T)$, equipped with the Hasudorff distance,
is a complete metric space (see \cite{kechris-louveau}, IV.2).

Following the theme and notation of K\"{o}rner \cite{korner:curve}, we now introduce
the metric space suitable for our purpose, as follows. Being given a number $q > 2$,
we define $\G_q$ to be the collection of all ordered pairs $(K, S)$ such that
$K \in \K(\T)$ and $S$ is a distribution in $A_q$ supported by $K$. We equip the
space $\G_q$ with the metric
\begin{equation}
\label{eq:gq-metric}
d((K,S), (K', S')) = \delta(K, K') + \|S - S'\|_{A_q}.
\end{equation}
It is obvious that $\G_q$ is a metric space. Moreover, we have

\begin{proposition}
\label{prop:gq-complete}
The metric space $\G_q$ is complete.
\end{proposition}

The fact that $\G_q$ is a complete metric space justifies the use of Baire
category. The result of this section is that Piatetski-Shapiro's phenomenon is
``typical'' in the space $\G_q$. Precisely, we prove:

\begin{theorem}
\label{thm:piatetski-theorem-lq-baire}
For quasi-all pairs $(K, S) \in \G_q$, the following holds:
\begin{enumerate-math}
\item The distribution $S$ supported by $K$ is non-zero.
\item The compact $K$ supports no non-zero measure $\mu \in A_q$.
\end{enumerate-math}
\end{theorem}

\subsection{Basic facts}
We first describe some basic properties of the space $\G_q$ which will be needed.
Their proofs are essentially the same as those given in \cite{korner:curve}.
We start with the proof of the fact that $\G_q$ is complete.

\begin{proof}[Proof of Proposition \ref{prop:gq-complete}]
The product space $\K(\T) \times A_q(\T)$, endowed with the metric \eqref{eq:gq-metric},
is certainly a complete metric space. It is therefore enough to show that $\G_q$ is a
closed subspace of $\K(\T) \times A_q(\T)$.

Let $(K,S) \in \K(\T) \times A_q(\T)$, and suppose that $(K, S) \notin \G_q$.
Then $S$ is not supported by $K$, hence there exists a $C^\infty$ function
$\varphi: \T \to \C$ which vanishes on some open set $U \subset \T$ containing $K$,
and such that $\dotprod{S}{\varphi} \neq 0$. Choose $\eps > 0$ such that $K_\eps
\subset U$ and $\eps \|\varphi\|_{A_p} < |\dotprod{S}{\varphi}|$. We claim that the
open ball with center $(K,S)$ and radius $\eps$ is disjoint from $\G_q$. Indeed,
suppose that $(K',S') \in \K(\T) \times A_q(\T)$ is such that $d((K,S), (K',S')) < \eps$.
Then $\delta(K, K') < \eps$, which implies $K' \subset K_\eps$, and therefore $\varphi$
vanishes on an open set containing $K'$. Also $\|S - S'\|_{A_q} < \eps$, which implies
$|\dotprod{S}{\varphi} - \dotprod{S'}{\varphi}| < \eps  \|\varphi\|_{A_p}$, and
therefore $\dotprod{S'}{\varphi} \neq 0$. This shows that  $(K',S') \notin \G_q$. We
have thus confirmed that the complement of $\G_q$ is an open set, which proves the claim.
\end{proof}

We will use the following fact about the topology of $\G_q$.

\begin{lemma}
\label{lemma:gq-topology}
For any open set $U \subset \T$, the sets
\[
\{(K, S) \in \G_q : K \subset U\} \quad \text{and} \quad
\{(K, S) \in \G_q : \supp(S) \cap U \neq \emptyset\}
\]
are open sets in $\G_q$.
\end{lemma}

The proof of Lemma \ref{lemma:gq-topology} is based on similar arguments as
in the proof of Proposition \ref{prop:gq-complete}, so we omit it.

\begin{lemma}
\label{lemma:quasi-surely-support}
Quasi-all $(K,S) \in \G_q$ have $K = \supp(S)$.
\end{lemma}

\begin{proof}
Let $(K,S) \in \G_q$. If $\supp(S) \neq K$, then there is a closed arc $J \subset \T$
such that $K \cap J \neq \emptyset$, and such that $S$ vanishes on the interior of $J$.
Moreover, we can choose $J$ to be an arc with rational endpoints. The collection of
all pairs $(K,S) \in \G_q$ such that $\supp(S) \neq K$ is therefore contained in a
countable union of sets of the form
\[
A(J) = \{(K,S) \in \G_q : K \cap J \neq \emptyset, \;
\text{$S$ vanishes on the interior of $J$}\},
\]
where $J$ denotes a closed arc. It is therefore enough to prove that $A(J)$ is nowhere
dense, for any such $J$.

From Lemma \ref{lemma:gq-topology} it follows that $A(J)$ is closed, so we only need
to show that its complement is dense. To this end, let $(K,S) \in A(J)$ and $\eps > 0$
be given. Since $K \cap J \neq \emptyset$, we can choose a closed arc $I$, of length
smaller than $\eps/2$, which is contained in the interior of $J$ and such that $I \subset
K_{\eps/2}$. Define $K' = K \cup I$, then $\delta(K, K') \leq \eps/2$. Next choose a non-zero
$C^\infty$ function $\varphi: \T \to \C$, supported by $I$, such that $\|\varphi\|_{A_q}
< \eps/2$. Define $S' = S + \varphi$, then $S'$ does not vanish on the interiour of $J$,
and $\|S - S'\|_{A_q} < \eps/2$. We therefore found $(K',S') \in \G_q$ such that
$d((K,S), (K',S')) < \eps$ and $(K',S') \notin A(J)$. This proves that $A(J)$ is nowhere
dense.
\end{proof}

An immediate consequence of Lemma \ref{lemma:quasi-surely-support} is the fact that
$S \neq 0$ for quasi-all $(K,S) \in \G_q$.

\subsection{Proof of main result}

\begin{lemma}
\label{lemma:principal-lq-baire}
Let $(K,S) \in \G_q$. Given any $\eps > 0$ and any positive integer $\nu$,
we can find $(K',S') \in \G_q$, $d((K,S), (K',S')) < \eps$, such that the
following holds: there is a real trigonometric polynomial
\[
X(t) = \sum_{|n| \geq \nu} \ft{X}(n) e^{int},
\quad \|\ft{X}\|_p \leq 1, \quad
\text{$\tfrac1{100} \leq X(t) \leq 100$ on $K'$}.
\]
\end{lemma}

\begin{proof}
According to Lemma \ref{lemma:quasi-surely-support}, the collection
of all pairs $(K,S) \in \G_q$ such that $K = \supp(S)$ is residual. In
particular, by the Baire category theorem, this collection is dense.
We may therefore assume that $S \neq 0$ and $K = \supp(S)$.

We start with the same argument as in Lemma \ref{lemma:principal-lq-kaufman}.
Given $\eta > 0$, using Lemma \ref{lemma:principal-lq} we can find a
compact $E \subset \T$, a $C^\infty$ function $f: \T \to \C$ and a real
trigonometric polynomial $X$ such that
\begin{enumerate-math}
\item
$f$ is supported by $E$, $\|f - 1\|_{A_q} < \eta$.
\item
$X(t) = \sum_{|n| \geq \nu} \ft{X}(n) e^{int}$, $\|\ft{X}\|_p \leq 1$
and $\tfrac1{100} \leq X(t) \leq 100$ on $E$.
\end{enumerate-math}
Define $f_m(t) = f(mt)$, $X_m(t) = X(mt)$ and let $E_m$ denote the compact
defined by $\1_{E_m}(t) = \1_{E}(mt)$. Then for every positive integer $m$
the conditions (i) and (ii) are also satisfied by $E_m$, $f_m$ and $X_m$
instead of $E$, $f$ and $X$.

We claim that by choosing first a sufficiently small $\eta$, and then
a sufficiently large $m$, the required pair $(K', S')$ could be found by
taking $S'$ to be the product $S f_m$, and $K' = K \cap E_m$.

Indeed, to make the distance between $S$ and $S'$ small, we may argue as in
Lemma \ref{lemma:principal-lq-kaufman} that if $m$ is sufficiently large then
$\|S' - S\|_{A_q}$ will be arbitrarily close to $\|S\|_{A_q} \|f - 1\|_{A_q}$.
So we may assume $\|S' - S\|_{A_q} < \eta \|S\|_{A_q}$. In particular, if
$\eta$ is sufficiently small then $\|S' - S\|_{A_q} < \eps / 2$.

To make the distance between $K$ and $K'$ small, we argue as follows. First,
we choose a finite set $T \subset K$ such that $K \subset T_{\eps/4}$. Then,
since $K = \supp(S)$, for every $t \in T$ we can choose a $C^\infty$ function
$\varphi_t: \T \to \C$, supported by $\big(t - \frac{\eps}{4}, t + 
\frac{\eps}{4}\big)$, such that $\dotprod{S}{\varphi_t} \neq 0$. We have
\[
|\dotprod{S'}{\varphi_t} - \dotprod{S}{\varphi_t}| =
|\dotprod{S' - S}{\varphi_t}| < \eta \|S\|_{A_q} \|\varphi_t\|_{A_p},
\]
so if $\eta$ is sufficiently small, $\dotprod{S'}{\varphi_t} \neq 0$ for
every $t \in T$ (note that $T$ was chosen before $\eta$). It follows that
$\supp(S')$ intersects each of the segments
$\big(t - \frac{\eps}{4}, t + \frac{\eps}{4}\big)$, $t \in T$. Since
$\supp(S') \subset K'$, this shows that $T \subset K'_{\eps/4}$. It follows
that $K \subset K'_{\eps / 2}$. On the other hand, $K' \subset K$, so we
conclude that $\delta(K', K) \leq \eps / 2$.

We have proved that for an appropriate choice of $\eta$ and $m$ we have
$\|S' - S\|_{A_q} < \eps / 2$ and $\delta(K', K) \leq \eps / 2$, hence $d((K,S),
(K',S')) < \eps$ as needed.
\end{proof}

\begin{proof}[Proof of Theorem \ref{thm:piatetski-theorem-lq-baire}]
For any positive integer $\nu$, let $W_\nu$ denote the set of all pairs
$(K,S) \in \G_q$ for which there exists a real trigonometric polynomial
\begin{equation}
\label{eq:nomeasure-lq-condition}
X(t) = \sum_{|n| \geq \nu} \ft{X}(n) e^{int}, \quad
\|\ft{X}\|_{p} \leq 1, \quad \text{$\frac1{200} < X(t) < 200$ on $K$}.
\end{equation}
From Lemma \ref{lemma:gq-topology} it follows that $W_\nu$ is an open
set, while Lemma \ref{lemma:principal-lq-baire} implies that $W_\nu$
is dense. It follows that the intersection $\bigcap_{\nu=1}^{\infty}
W_\nu$ is a residual set. That is, quasi-all pairs $(K,S) \in \G_q$
satisfy the following condition: for every $\nu$ there is a real
trigonometric polynomial $X(t)$ satisfying \eqref{eq:nomeasure-lq-condition}.
We may therefore use Lemma \ref{lemma:nomeasure-lq} (the precise constant
which appears there is different, but clearly plays no special role),
which implies that quasi-surely $K$ does not support a non-zero measure
$\mu \in A_q$. On the other hand, Lemma \ref{lemma:quasi-surely-support}
ensures that $S \neq 0$ quasi-surely, so the theorem is proved.
\end{proof}

\section{$p$-Helson sets}

\subsection{}

A compact $K$ on the circle $\T$ is called a \emph{$p$-Helson set} if any
function, defined and continuous on $K$, can be extended to a continuous
function on the circle with Fourier coefficients belonging to $\ell_p$.

When $p=1$ these are the usual Helson sets, which have been much studied
(see \cite{kahane:absolument}). Helson \cite{helson} proved that such a set
cannot support a measure with Fourier coefficients tending to zero. On the
other hand, K\"{o}rner \cite{korner:pseudofunc} proved that a Helson set can
support a \emph{distribution} with this property, that is, there exists a
Helson set of multiplicity. Kaufman \cite{kaufman:msets} strengthened the
result by proving that any set of multiplicity contains a Helson set of
multiplicity.

We now turn to the case $1 < p < 2$. Olevskii \cite{olevskii:localization,
olevskii:general} constructed a compact set of measure zero, which is not
$p$-Helson for any $p < 2$. On the other hand Demenko \cite{demenko:plane,
demenko:ndim} proved that if $K$ is sufficiently small in the sense of Hausdorff
dimension, or precisely, if $\dim K < 2 - 2/p$, then $K$ is a $p$-Helson set.
We point out that no analog is true for $p=1$; there exist countable compacts
which are not Helson sets. An analog of Helson's theorem was proved by Gregory
in \cite{gregory}, and by Demenko in \cite{demenko:plane}. Precisely, they
proved that a $p$-Helson set $(1 < p < 2)$ cannot support a measure with
Fourier transform belonging to $\ell_q$, where $q = p/(p-1)$ is the conjugate
of $p$. We remark that a $p$-Helson set can support a measure with Fourier
transform tending to zero, as shown in \cite{gregory}.

Could a $p$-Helson set support a \emph{distribution} with Fourier transform
belonging to $\ell_q$? In this section we answer this question affirmatively.
We prove the following strengthening of Theorem \ref{thm:piatetski-theorem-lq},
analogous to K\"{o}rner's theorem about the existence of a Helson set of multiplicity.

\begin{theorem}
\label{thm:p-helson-multiplicity}
For any $1 < p < 2$ there exists a $p$-Helson set which supports a non-zero
distribution with Fourier coefficients belonging to $\ell_q$, $q = p / (p-1)$.
\end{theorem}

\subsection{}
Let $p$ be a fixed number, $1 < p < 2$, and let $q = p/(p-1)$ be the exponent
conjugate of $p$. We use the following characterization of $p$-Helson sets,
which is due to Gregory \cite{gregory}.

\begin{proposition}[Gregory]
\label{prop:p-helson-gregory}
A compact $K$ is a $p$-Helson set if and only if $|\mu|(K) = 0$ for every measure
$\mu \in A_q$.
\end{proposition}

We also need the following lemma:

\begin{lemma}
\label{lemma:measure-abs-restrict-lq-approx}
Let $K$ be a compact on the circle, and $\mu$ be a measure in $A_q(\T)$. Then
the measure $\1_K \cdot |\mu|$ belongs to the closure of $A_q \cap M(\T)$ in the
$M(\T)$ norm.
\end{lemma}

The proof of Lemma \ref{lemma:measure-abs-restrict-lq-approx} is very similar
to that of Lemma \ref{lemma:measure-abs-lq-approx}, so it is omitted.

\begin{lemma}
\label{lemma:p-helson-condition}
Suppose a compact $K$ on the circle satisfies the following
condition: for any $\eps > 0$ there is a real trigonometric polynomial
$X$ such that 
\[
\|X\|_{\infty} \leq 1, \quad \|X\|_{A_p} < \eps, \quad
\text{$X(t) > \tfrac1{50}$ on $K$}.
\]
Then $K$ is a $p$-Helson set.
\end{lemma}

\begin{proof}
Let $\mu \in A_q \cap M(\T)$. Given $\eps > 0$, by Lemma
\ref{lemma:measure-abs-restrict-lq-approx} one can find a measure
$\mu_1 \in A_q(\T)$ such that
\[
\big\| \; \1_K \cdot |\mu| -  \mu_1 \; \big\|_{M(\T)} < \eps.
\]
Our assumption on the compact $K$ now allows us to choose a real trigonometric
polynomial $X$ satisfying $\|X\|_{\infty} \leq 1$, $\|X\|_{A_p} \|\mu_1\|_{A_q} < \eps$,
$X(t) > \tfrac1{50}$ on $K$. Then
\[
\Big| \int_{\T} X d\mu_1 \Big| \geq 
\int_{K} X |d\mu| - \eps \|X\|_{\infty} \geq
\tfrac1{50} |\mu| (K) - \eps.
\]
On the other hand,
\[
\Big| \int_{\T} X d\mu_1 \Big| = 
\Big| \sum_{n \in \Z} \ft{X}(n) \ft{\mu}_1(-n) \Big| \leq
\|X\|_{A_p} \|\mu_1\|_{A_q} < \eps.
\]
It follows that $|\mu|(K) = 0$, for every measure $\mu \in A_q$. By Proposition
\ref{prop:p-helson-gregory}, this is equivalent to the fact that $K$ is a $p$-Helson set.
\end{proof}

\begin{proof}[Proof of Theorem \ref{thm:p-helson-multiplicity}]
For a sequence $\{\eps_j\}$ tending to zero, we choose $K_j$, $F_j$ and $X_j$
according to Lemma \ref{lemma:principal-generators}. As before, we can choose
the $\eps_j$ by induction, in such a way that the product $\prod_{j=1}^{\infty}
F_j$ will converge in $A_q$ to a non-zero distribution $S$ supported by
$K := \bigcap_{j=1}^{\infty} K_j$. Lemma \ref{lemma:p-helson-condition} now
implies that $K$ is a $p$-Helson set, so the theorem is proved.
\end{proof}

\chapter{Orlicz Spaces}

In this chapter we show how the method developed in Chapter \ref{chap:lq}
can be adopted to further extend Piatetski-Shapiro's phenomenon to spaces of
sequences more general than $\ell_q$ spaces. As mentioned before, it is
known that no Piatetski-Shapiro phenomenon exists in certain weighted $\ell_2$
spaces. A substantial difference between $\ell_q$ spaces and weighted spaces
is that $\ell_q$ spaces are rearrangement invariant, while weighted spaces are
in general not. By ``rear\-range\-ment invariant space'' we mean a space $\X$
of sequences on $\Z$ such that, for every sequence $\{x_n\} \in \X$ and every
permutation (i.e. bijection) $\sigma: \Z \to \Z$, the sequence $\{x_{\sigma(n)}\}$
also belongs to $\X$.

Orlicz spaces, whose definion we recall below, is a well-known class of rearrangement
invariant spaces of sequences. They are therefore natural candidates for additional
spaces where Piatatski-Shapiro's phenomenon could exist. We show that for a certain
class of Orlicz spaces, this is indeed the case.


\section{Introduction}

\subsection{Orlicz spaces}
Let us start by recalling the basics of Orlicz sequence spaces, following
\cite{linden-tzaf}. An Orlicz function $\phi$ is a continuous non-decreasing and
convex function on $[0, \infty)$ such that $\phi(0) = 0$ and $\lim_{t \to \infty}
\phi(t) = \infty$. To any Orlicz function $\phi$ there corresponds the so-called
Orlicz space $\ell_\phi$ consisting of all sequences of scalars
$x = \{x_n\}_{n \in \Z}$ such that $\sum_{n \in \Z} \phi(|x_n| / \rho) < \infty$
for some $\rho > 0$. The space $\ell_\phi$ equipped with the norm
\begin{equation}
\label{eq:orlicz-norm}
\|x\|_\phi = \inf \{\rho > 0 : \sum_{n \in \Z} \phi(|x_n| / \rho) \leq 1\}
\end{equation}
becomes a Banach space. Orlicz spaces generalize $\ell_p$ spaces in a natural way,
since $\ell_p$ is the Orlicz space which corresponds to $\phi(t) = t^p$.

If $\phi(t_0) = 0$ for some $t_0 > 0$, that is, $\phi$ vanishes in some neighborhood
of zero, we say that $\phi$ is a \emph{degenerate} Orlicz function. For a degenerate
Orlicz function $\phi$, as can be easily checked, the space $\ell_\phi$ coincides with
$\ell_\infty$. Since this case is not of interest in our context we shall consider from
now on only non-degenerate Orlicz functions.

Two Orlicz functions $\phi_1$, $\phi_2$ are said to be \emph{equivalent at zero}
if there exist constants $C$, $M$ and $t_0 > 0$ such that, for all $0 \leq t \leq t_0$,
\[
C^{-1} \, \phi_2(M^{-1} t) \leq \phi_1(t) \leq C \, \phi_2(Mt).
\]
In this case, $\ell_{\phi_1} = \ell_{\phi_2}$ (i.e. both spaces consist of the same sequences)
and the norms $\|\cdot\|_{\phi_1}$ and $\|\cdot\|_{\phi_2}$ are equivalent.

An Orlicz function $\phi$ is said to satisfy the $\deltatwo$-condition at zero if
\begin{equation}
\label{eq:delta-two-condition}
\limsup_{t \to 0} \frac{\phi(2 t)}{\phi(t)} < \infty.
\end{equation}
It is easily verified that \eqref{eq:delta-two-condition} implies that
$\limsup_{t \to 0} \phi(M t) / \phi(t) < \infty$ for every $M > 0$.
The role of the $\deltatwo$-condition is illustrated by the following:

\begin{proposition}[{see \cite{linden-tzaf}, Proposition 4.a.4}]
\label{prop:deltatwo-equivalent}
For an Orlicz function $\phi$ the following conditions are equivalent:
\begin{enumerate-math}
\item $\phi$ satisfies the $\deltatwo$-condition at zero.
\item $\ell_\phi$ is separable.
\item
The system of unit vectors $e_n = (\dots, 0, \overset{n}{1}, 0, \dots)$, $n \in \Z$,
is complete in $\ell_\phi$.
\end{enumerate-math}
\end{proposition}


\subsection{Result}
If no further conditions on $\phi$ are imposed, then it is easy to exhibit
spaces $\ell_\phi$ where Piatetski-Shapiro's phenomenon does exist. For let
$K$ be the compact given by Theorem \ref{thm:piatetski-theorem-c0} and $S$ be a
distribution supported by $K$ such that $\ft{S} \in c_0$. Then one may construct
a (non-degenerate) Orlicz function $\phi$ such that $\ft{S} \in \ell_\phi$ (for
this it will be enough that $\phi$ has a sufficiently fast decay at zero).
Since $\phi$ is non-degenerate we have $\ell_\phi \subset c_0$, and it follows
immediately that Piatetski-Shapiro's phenomenon exists in $\ell_\phi$.

Our last remark shows that if $\phi$ decays sufficiently fast at zero, then
Piatetski-Shapiro's phenomenon does exist in $\ell_\phi$. On the other hand, if
$\phi$ decays sufficiently slowly at zero then no Piatetski-Shapiro phenomenon can
exist in $\ell_\phi$. Precisely, if $\liminf_{t \to 0} \phi(t) / t^2 > 0$ then
$\ell_\phi \subset \ell_2$, hence any distribution $S$ such that $\ft{S} \in
\ell_\phi$ is automatically a measure (in fact, an $L^2$ function).

In order to give non-trivial conditions under which Piatetski-Shapiro's phenomenon
exists in $\ell_\phi$, we introduce the following definition. We say that an Orlicz
function $\phi$ is \emph{submultiplicative at zero} if there exist constants $M$
and $t_0 > 0$ such that, for every $0 < s,t < t_0$ one has
\begin{equation}
\label{eq:orlicz-submultiplicative}
\phi(st) \leq M \phi(s) \phi(t).
\end{equation}
Submultiplicative Orlicz functions appear in the theory of operators in Orlicz
spaces, see \cite{krasno-rutickii}. We can now formulate the result of this chapter.

\begin{theorem}
\label{thm:piatetski-theorem-orlicz}
Let $\phi$ be a non-degenerate Orlicz function satisying the following
conditions: (i) $\lim_{t \to 0} \phi(t)/t^2 = 0$; (ii) $\phi$ satisfies the
$\deltatwo$-condition at zero; and (iii) $\phi$ is submultiplicative at zero.
Then there is a compact $K$, which supports a distribution $S$ such that
$\ft{S} \in \ell_\phi$, but which does not support such a measure.
\end{theorem}

\begin{remarks}
\item
As mentioned above, the condition that $\phi$ be non-degenerate is
required in order to exclude the case $\ell_\phi = \ell_\infty$.

\item
The condition (i) can be relaxed a bit. In fact, as we will show, the
theorem remains true if (i) is replaced by the weaker condition (i)$'$
defined by
\begin{equation}
\label{eq:condition-theorem-orlicz-relaxed}
\text{(i)$'$} \quad \liminf_{t \to 0} \phi(t)/t^2 = 0 \quad
\text{and} \quad \limsup_{t \to 0} \phi(t)/t^2 < \infty.
\end{equation}
It can be checked that (i)$'$ holds if and only if there exists a
constant $C$ such that $\| \cdot \|_\phi \leq C \| \cdot \|_2$, but
the norms are not equivalent. 
As explained above, the condition $\liminf_{t \to 0} \phi(t)/t^2 = 0$
is crucial in order to exclude the case $\ell_\phi \subset \ell_2$.


\item
We leave open the question whether the theorem remains true in the
case $\limsup_{t \to 0} \phi(t)/t^2 = \infty$, and whether
the conditions (ii) and (iii) could be relaxed or not.


\item
By choosing the function $\phi(t) = t^q$ $(q > 2)$ one can see that
Theorem \ref{thm:piatetski-theorem-orlicz} is in fact a strengthening of 
Theorem \ref{thm:piatetski-theorem-lq}, since in this case we have
$\ell_\phi = \ell_q$.
\end{remarks}

%
%

\subsection{Example}
Let us explicitly describe a family of Orlicz function to which Theorem
\ref{thm:piatetski-theorem-orlicz} can be applied. Fix two numbers $q > 2$
and $\alpha > 0$. Define
\begin{equation}
\label{eq:orlicz-example}
\phi(t) = t^q \log^{\alpha} \frac1{t}
\end{equation}
for $t > 0$, and $\phi(0) = 0$.
It can be checked that $\phi$ is continuous, non-decreasing and convex on
some interval $[0, t_0]$. It is therefore possible to extend $\phi$ to the
whole segment $[0, \infty)$ so that it becomes an Orlicz function.
Note that the corresponding space $\ell_\phi$ will be the same regardless of how
$\phi$ was extended, and the norms associated with two distinct extensions
might be different but equivalent \cite[p. 139]{linden-tzaf}. It can
now be checked directly using \eqref{eq:orlicz-example} that the function
$\phi$ satisfies the conditions of Theorem \ref{thm:piatetski-theorem-orlicz}.
It can also be checked that a sequence $\{x_n\}$ belongs to $\ell_\phi$ if and
only if $\sum' |x_n|^q \log^{\alpha} (e + |x_n|^{-1}) < \infty$, where the
notation $\sum'$ is used to indicate summation only over indices $n$
such that $x_n \neq 0$. In view of these remarks, the conclusion of Theorem
\ref{thm:piatetski-theorem-orlicz} in this case can be stated as follows.

\begin{corollary}
For any $q > 2$ and any $\alpha > 0$ there is a compact $K$ on the circle,
which supports a distribution $S$ such that
\[
\sideset{}{'}\sum_{n \in \Z} |\ft{S}(n)|^q \log^{\alpha} (e + |\ft{S}(n)|^{-1}) < \infty,
\]
but does not support such a measure.
\end{corollary}


\section{Lemmas}

\subsection{Notation}
From now on we assume that $\phi$ is a non-degenerate Orlicz function
satisfying conditions (ii) and (iii) of Theorem \ref{thm:piatetski-theorem-orlicz},
while instead of condition (i) we will assume the weaker condition (i)$'$
given by \eqref{eq:condition-theorem-orlicz-relaxed}.
We denote
by $A_\phi$ the space of Schwartz dis\-tributions $S$ on the circle $\T$
such that $\ft{S} \in \ell_\phi$. Equipped with the norm $\|S\|_{A_\phi} :=
\|\ft{S}\|_{\phi}$ the space $A_\phi$ is clearly a Banach space.


It will be convenient to extend the definition of the Orlicz norm
$\| \cdot \|_{\phi}$ to sequences defined not only on $\Z$, but on
any finite or countable set $I$. Precisely, given a sequence of scalars
$\{x_n\}_{n \in I}$ we define
\begin{equation}
\label{eq:orlicz-norm-finite}
\|\{x_n\}\|_{\phi} = \inf \{\rho > 0 : \sum_{n \in I} \phi(|x_n| / \rho) \leq 1\}.
\end{equation}

In what follows $r$ will denote a positive (small) number. Given $r > 0$,
we let $N = N(r)$ denote the unique integer satisying
\begin{equation}
\label{eq:def-n-orlicz}
\frac1{\phi(r)} \leq N < \frac1{\phi(r)} + 1.
\end{equation}
Note that $N \to \infty$ as $r \to 0$, since $\phi$ is continuous
and $\phi(0)=0$.

\subsection{Excluding measures}

In Lemma \ref{lemma:nomeasure-lq} we considered a trigonometric
polynomial $X(t)$ whose Fourier coefficients satisfy $\|\{\ft{X}(n)\}\|_p
\leq 1$. This was used in order to deduce that
\[
\Big| \sum c(n) \ft{X}(n) \Big| \leq \|\{c(n)\}\|_q
\]
for any choice of scalars $\{c(n)\}$, where $p^{-1} + q^{-1} = 1$.
In the context of Orlicz spaces, the analog claim is related to the
description of the dual space of $\ell_\phi$. It is known
\cite[pp. 147--148]{linden-tzaf} that if $\phi$ satisfies the
$\deltatwo$-condition at zero then the dual space $\ell^{\, *}_\phi$ may be
identified with another Orlicz space $\ell_\psi$, where $\psi$ is a certain
Orlicz function ``complementary'' to $\phi$. However, we shall not need
such a result. The following lemma, which can be proved directly, would be
sufficient for us.

\begin{lemma}
\label{lemma:orlicz-dual-estimate}
Given $r > 0$, let $N = N(r)$ be defined by \eqref{eq:def-n-orlicz}.
Then for every choice of scalars $c_1, \dots, c_N$ one has
\begin{equation}
\label{eq:orlicz-dual-estimate}
\Big| \frac1{N r} \sum_{n=1}^{N} c_n \Big| \leq \|\{c_n\}\|_{\phi}
\end{equation}
\end{lemma}

\begin{proof}
Since both sides of \eqref{eq:orlicz-dual-estimate} are homogeneous,
it would be enough to prove the claim in the case $\|\{c_n\}\|_{\phi} = 1$.
According to the definition \eqref{eq:orlicz-norm-finite} of the Orlicz
norm, and due to the continuity of $\phi$, this is equivalent to $\sum_{n=1}^{N}
\phi(|c_n|) = 1$. By the convexity of $\phi$, and using \eqref{eq:def-n-orlicz},
we have
\begin{equation}
\label{eq:inequality-convexity-phi}
\phi \Big( \frac1{N} \sum_{n=1}^{N} |c_n| \Big) \leq
\frac1{N} \sum_{n=1}^{N} \phi(|c_n|) = \frac1{N} \leq \phi(r).
\end{equation}
However, $\phi$ is strictly increasing, since it is non-degenerate, so from
\eqref{eq:inequality-convexity-phi} it follows that $\frac1{N} \sum_{n=1}^{N}
|c_n| \leq r$. That is,
\[
\Big| \frac1{N r} \sum_{n=1}^{N} c_n \Big| \leq 1,
\]
as we had to show.
\end{proof}

With the aid of Lemma \ref{lemma:orlicz-dual-estimate} we may now
prove an analog of Lemma \ref{lemma:nomeasure-lq}.

\begin{lemma}
\label{lemma:nomeasure-orlicz}
Let $K$ be a compact on the circle. Suppose that for any positive
integer $\nu$ there exists $r > 0$ such that the following holds: the
trigonometric polynomial defined by
\[
X(t) = \frac1{N r} \sum_{j=1}^{N} \cos \nu^j t,
\]
where $N = N(r)$ is defined by \eqref{eq:def-n-orlicz},
satisfies $\tfrac1{100} \leq X(t) \leq 100$ on $K$.
Then $K$ does not support a measure $\mu \in A_\phi$.
\end{lemma}

\begin{proof}
Suppose that $\mu \in A_\phi \cap M(K)$.
As in Lemma \ref{lemma:measure-abs-lq-approx}, given $\eps > 0$ we can
approximate $|\mu|$ in the $M(K)$ norm by a measure $\mu_1 \in
A_\phi$. We have
\[
\int_{\T} X d\mu_1 = \sum_{n \in \Z} \ft{X}(n) \ft{\mu}_1(-n)
= \frac1{Nr} \sum_{j=1}^{N} \frac{\ft{\mu}_1(\nu^j)
+ \ft{\mu}_1(-\nu^j)}{2} \, ,
\]
hence it follows by Lemma \ref{lemma:orlicz-dual-estimate} that
\begin{equation}
\label{eq:nomeasure-estimate-orlicz}
\Big| \int_{\T} X d\mu_1 \Big| \leq
\Big\| \Big\{ \frac{\ft{\mu}_1(\nu^j)
+ \ft{\mu}_1(-\nu^j)}{2} \Big\}_{j=1}^{N} \Big\|_{\phi} \, .
\end{equation}
We now use the $\deltatwo$-condition:
according to Proposition \ref{prop:deltatwo-equivalent}, the
system of unit vectors $\{e_n\}_{n \in \Z}$ is complete in $\ell_\phi$,
hence the Orlicz norm of the ``tail'' sequence
$\{\ft{\mu}_1(n)\}_{|n| \geq \nu}$ is arbitrarily small if $\nu$
is sufficiently large. From this it is easy to deduce that the
quantity on the right side of \eqref{eq:nomeasure-estimate-orlicz}
is smaller than $\eps$ for suffi\-ciently large $\nu$. Once this has
been established, the proof can then be finished as in Lemma
\ref{lemma:nomeasure-lq}.
\end{proof}

\subsection{Riesz products}
The appropriate Riesz product in this context is
\[
\lambda_s(t) = \prod_{j=1}^{N} \big(1 + 2 s r \cos \nu^j t \big),
\]
where $0 < r < 1$, $N = N(r)$ is defined by \eqref{eq:def-n-orlicz},
$\lowers < s < \uppers$ and $\nu \geq 3$. As before, this
defines a probability measure $\lambda_s$ on the circle $\T$ such that
\begin{equation}
\label{eq:lambda-coeff-orlicz}
\ft{\lambda}_s \Big(\sum_{j=1}^{N} \tau_j \nu^j\Big) = 
(sr)^{\sum |\tau_j|} \; , \quad
\bar\tau = (\tau_1,\dots,\tau_N) \in \{-1,0,1\}^N.
\end{equation}
Our corresponding trigonometric polynomial $X(t)$ will be
defined by
\begin{equation}
\label{eq:x-cosines-orlicz}
X(t) = \frac1{N r} \sum_{j=1}^{N} \cos \nu^j t,
\end{equation}
and a straightforward calculation shows again that the expectation of $X$
with respect to the measure $\lambda_s$ is equal to $s$. We have the 
following lemma, whose proof we omit, as it is basically the same as the
proof of Lemma \ref{lemma:lambda-concentration-lq}.

\begin{lemma}
\label{lemma:lambda-concentration-orlicz}
Being given $0 < r < 1$, for $\nu \geq \nu(r)$ one has
\begin{equation}
\label{eq:lambda-concentration-orlicz}
\lambda_s \{t \in \T : |X(t) - s| > \alpha\} 
\leq 3 \exp\big(- \tfrac1{8} \alpha^2 N r^2 \big)
\end{equation}
for every $\alpha > 0$ and every $\lowers < s < \uppers$.
\end{lemma}

%
%

The next lemma is an analog of Lemma \ref{lemma:lambda-restriction-lq}.
Note however that the estimate \eqref{lemma:lambda-restriction-rough-estimate}
is given explicitly in terms of the function $\phi$. This will allows us to
prove the result under the weaker condition (i)$'$ instead of (i).

\begin{lemma}
\label{lemma:lambda-restriction-orlicz}
Denote
\[
K' = \big\{ t \in \T \; : \; \tfrac1{90} \leq X(t) \leq
90 \big\}.
\]
Then for $\nu \geq \nu(r)$ and for every $\lowers < s < \uppers$,
\begin{equation}
\label{lemma:lambda-restriction-rough-estimate}
\int_{\T \setminus K'} \lambda_s^2(t) \; \frac{dt}{2\pi} < 2^{11} \cdot
\frac{\phi(r)}{r^2}
\end{equation}
where $C > 0$ is an absolute constant.
\end{lemma}

\begin{proof}
We follow the same line as in the proof of Lemma \ref{lemma:lambda-restriction-lq},
but somewhat improve on one of our estimates. We have
\[
\lambda_s(t) = \prod_{j=1}^{N} \big(1 +  2 s r \cos \nu^j t \big)
\leq \exp \Big( 2 s r \sum_{j=1}^{N}
\cos \nu^j t \Big),
\]
so using \eqref{eq:x-cosines-orlicz} we get
\begin{equation}
\label{eq:lambda-max-orlicz}
\lambda_s(t) \leq \exp \left( 2 s N r^2 X(t) \right).
\end{equation}

We now apply Lemma \ref{lemma:lambda-concentration-orlicz}, and for
$\nu \geq \nu(r)$ we use \eqref{eq:lambda-concentration-orlicz} and
\eqref{eq:lambda-max-orlicz}. The same argument as in the proof of Lemma
\ref{lemma:lambda-restriction-lq} will show that
\[
\int_{\big\{t \; : \; X(t) < \frac1{90}\big\}} \lambda_s^2(t) \; \frac{dt} {2\pi}
\leq 3 \exp \left(- 2^{-10} N r^2 \right).
\]
Also for any integer $90 \leq k \leq 1/r$,
\[
\begin{aligned}
&\int_{\big\{t \; : \; k < X(t) \leq k+1\big\}} \lambda_s^2(t) \;
\frac{dt}{2\pi}
\leq \int_{\big\{t \; : \; k < X(t) \leq k+1\big\}} \lambda_s(t) \;
\frac{dt}{2\pi}
\cdot \max_{\big\{t \; : \; k < X(t) \leq k+1\big\}} \lambda_s(t)\\[4pt]
&\qquad \qquad \leq 3 \exp \left( -\tfrac1{8}(k - s)^2 N r^2 \right) \cdot
\exp \left(2 s N r^2 (k+1) \right)\\[6pt]
&\qquad \qquad \leq 3 \exp \left(- \tfrac1{10} k^2 N r^2 \right),
\end{aligned}
\]
and so, since $X(t) \leq 1/r$ for every $t$,
\[
\begin{aligned}
\int_{\big\{t \; : \; X(t) > 90\big\}} \lambda_s^2(t) \; \frac{dt}{2\pi}
&\leq 3 \sum_{90 \leq k \leq 1/r} \exp \left(- \tfrac1{10} k^2 N r^2 \right)\\
&\leq 3 \int_{0}^{\infty} x \cdot \exp \left(- \tfrac1{10} x^2 N r^2 \right) \, dx
= \frac{15}{N r^2} \, .
\end{aligned}
\]
We thus get
\[
\int_{\T \setminus K'} \lambda_s^2(t) \; \frac{dt}{2\pi}
\leq 3 \exp \left(- 2^{-10} N r^2 \right) + \frac{15}{N r^2}
< \frac{2^{11}}{N r^2} \, .
\]
Using \eqref{eq:def-n-orlicz}, the result follows.
\end{proof}

\begin{remark}
As one can see from the above proof, the estimate
\eqref{lemma:lambda-restriction-rough-estimate} is very rough.
However, it will be enough for our purpose.
\end{remark}


\section{Proof of main result}

To finish the proof of Theorem \ref{thm:piatetski-theorem-orlicz}
it will be enough to establish the following analog of Lemma
\ref{lemma:principal-lq}. Once this is done, the theorem
follows by an iteration procedure as in the proof of Theorem
\ref{thm:piatetski-theorem-lq}, which we shall not repeat.

\begin{lemma}
\label{lemma:principal-orlicz}
Let $\eps > 0$. There exists $r = r(\eps)$, $0 < r < \eps$
such that the following holds: for any $\nu \geq \nu(r)$ there is a
$C^\infty$ function $f: \T \to \C$ satisfying
\begin{enumerate-math}
\item
$f$ is supported by
\begin{equation}
\label{eq:k-orlicz}
K = \big\{ t \in \T \; : \; \tfrac1{100} \leq X(t) \leq 100\big\}
\end{equation}
where $X$ is the trigonometric polynomial defined by
\eqref{eq:x-cosines-orlicz}.
\item
$f(t) = 1 + \sum_{n \neq 0} \ft{f}(n)e^{int}$,
where 
$\|\{\ft{f}(n) : n \neq 0\}\|_{\phi} < \eps$.
\end{enumerate-math}
\end{lemma}

\begin{proof}
As in the proof of Lemma \ref{lemma:principal-lq}, given $\delta > 0$
we choose a measure $\rho$ supported by the interval $(\lowers, \uppers)$
such that \eqref{eq:kahane-lemma} holds. We then define
\[
\lambda(t) = \int \lambda_s(t) \, d\rho(s),
\]
and so by \eqref{eq:lambda-coeff-orlicz}
\[
\lambda(t) = 1 + \sum_{\tau \neq 0}
\Big\{ r^{\sum |\tau_j|} \int s^{\sum |\tau_j|} d\rho(s) \Big\}
\; e^{i (\sum \tau_j \nu^j) t}.
\]
It follows that
\begin{equation}
\label{eq:norm-approx-one-orlicz-1}
\sum_{n \neq 0} \phi(|\ft{\lambda}(n)| / \delta)
\leq \sum_{0 \neq \bar{\tau} \in \{-1,0,1\}^N} \phi(r^{\sum |\tau_j|}).
\end{equation}
We now use the submultiplicativity at zero: there exists $M$ and
$t_0 > 0$ such that $\phi(st) \leq M \phi(s) \phi(t)$ for every
$s,t < t_0$. It is easy to check that this implies
\begin{equation}
\label{eq:orlicz-submult-powers}
\phi(r^k) \leq M^{k-1} \phi(r)^k
\end{equation}
for all sufficiently small $r$. We will assume that $M \geq 1$,
as we clearly may. We then apply \eqref{eq:orlicz-submult-powers} to
\eqref{eq:norm-approx-one-orlicz-1} to get
\[
\begin{aligned}
\sum_{n \neq 0} \phi(|\ft{\lambda}(n)| / \delta)
\leq \sum_{\bar\tau \in \{-1,0,1\}^N} (M \phi(r))^{\sum |\tau_j|}
= (1 + 2M \phi(r))^N < e^{2M N \phi(r)}.
\end{aligned}
\]
From \eqref{eq:def-n-orlicz} we have $N \phi(r) < 1 + \phi(r)$, and
since $\phi$ is continuous and $\phi(0) = 0$ this shows that for
sufficiently small $r$, $N \phi(r) \leq 2$. So we arrive at the estimate
\begin{equation}
\label{eq:norm-approx-one-orlicz-2}
\sum_{n \neq 0} \phi(|\ft{\lambda}(n)| / \delta) \leq e^{4M}.
\end{equation}
Observe that since $\phi$ is convex, $\phi(0) = 0$ and $e^{4M} > 1$,
we have $\phi(e^{-4M} x) \leq e^{-4M} \phi(x)$ for any $x > 0$. 
Thus using \eqref{eq:norm-approx-one-orlicz-2},
\[
\sum_{n \neq 0} \phi(|\ft{\lambda}(n)| e^{-4M} / \delta)
\leq e^{-4M} \sum_{n \neq 0} \phi(|\ft{\lambda}(n)| / \delta) \leq 1,
\]
and according to the definition of the Orlicz norm we conclude that
\begin{equation}
\label{eq:norm-approx-one-orlicz-3}
\|\lambda - 1\|_{A_\phi} \leq e^{4M} \delta.
\end{equation}
We now use the conditions
\[
\liminf_{t \to \infty} \phi(t)/t^2 = 0
\quad \text{and} \quad
\limsup_{t \to \infty} \phi(t)/t^2 < \infty.
\]
First, it is easy to check that from
$\limsup_{t \to \infty} \phi(t)/t^2 < \infty$
it follows that a constant $C$ exists such that
$\|\{x_n\}\|_{\phi} \leq C \|\{x_n\}\|_{2}$
for any sequence $\{x_n\} \in \ell_2$. Second, from
$\liminf_{t \to \infty} \phi(t)/t^2 = 0$ it follows
that there exist arbitrarily small values of $r$ such that
\[
\left( 2^{11} \cdot \frac{\phi(r)}{r^2} \right)^{1/2}
< \frac{\delta}{C \; \|\rho\|_M}.
\]
In particular, we may assume that $r$ is small enough so that
\eqref{eq:norm-approx-one-orlicz-3} holds. Using Lemma
\ref{lemma:lambda-restriction-orlicz}, it follows that for
$\nu \geq \nu(r)$
\[
\left( \int_{\T \setminus K'} \lambda_s^2(t) \; \frac{dt}{2\pi}
\right)^{1/2} < \frac{\delta}{C \; \|\rho\|_M}
\]
for every $\lowers < s < \uppers$. The function $h := \lambda \cdot \1_{K'}$
is then supported by $K'$, and
\[
\|\lambda - h\|_{A_\phi} \leq C \|\lambda - h\|_{L^2(\T)}
= C \|\lambda\|_{L^2(\T \setminus K')}
\leq C \int \|\lambda_s\|_{L^2(\T \setminus K')} \; |d\rho(s)|
< \delta.
\]
We then continue as in the proof of Lemma \ref{lemma:principal-lq}.
\end{proof}


\backmatter


\begin{thebibliography}{99}
\addtolength{\itemsep}{2pt}

\bibitem{bary}
N. K. Bary,
``A treatise on trigonometric series'',
Pergamon Press, 1964.

\bibitem{besicovitch}
A. S. Besicovitch,
``On the sum of digits of real numbers represented in the dyadic system'',
Math. Ann. \textbf{110} (1935), 321--330. 

\bibitem{beurling}
A. Beurling,
``On a closure problem'',
Ark. Mat. \textbf{1} (1951), 301--303.

\bibitem{boas}
R. P. Boas, Jr.,
``Entire functions'', Academic Press Inc., 1954.

\bibitem{debs-saint-raymond}
G. Debs and J. Saint-Raymond,
``Ensembles bor\'{e}liens d'unicit\'{e} et d'unicit\'{e} au sens large'',
Ann. Inst. Fourier (Grenoble) \textbf{37} (1987), 217--239.

\bibitem{demenko:plane}
V. N. Demenko,
``$p$-Helson curves in the plane'' (in Russian),
Mat. Zametki \textbf{39} (1986), 349--359.
English translation in Math. Notes \textbf{39} (1986), 189--195.

\bibitem{demenko:ndim}
V. N. Demenko,
``On $p$-Helson sets in $\R^n$'',
Anal. Math. \textbf{15} (1989), 17--35.

\bibitem{edwards}
R. E. Edwards,
``Spans of translates in $L^p(G)$'',
J. Austral. Math. Soc. \textbf{5} (1965), 216--233. 


\bibitem{graham-mcgehee}
C. C. Graham and O. C. McGehee,
``Essays in commutative harmonic analysis'',
Springer-Verlag, 1979.

\bibitem{gregory}
M. B. Gregory,
``$p$-Helson sets, $1<p<2$'', 
Israel J. Math. \textbf{12} (1972), 356--368. 


\bibitem{helson}
H. Helson,
``Fourier transforms on perfect sets'',
Studia Math. \textbf{14} (1954), 209--213.

\bibitem{herz}
C. S. Herz,
``A note on the span of translations in $L^p$'',
Proc. Amer. Math. Soc. \textbf{8} (1957), 724--727. 


\bibitem{hoeffding}
W. Hoeffding, 
``Probability inequalities for sums of bounded random variables'',
J. Amer. Statist. Assoc. \textbf{58} (1963), 13--30.


\bibitem{ivashev-musatov-2}
O. S. Ivashev-Musatov,
``M-sets and $h$-measures'' (Russian),
Mat. Zametki \textbf{3} (1968), 441--447.
English translation in Math. Notes \text{3} (1968), 281--285.

\bibitem{kahane:rearrangements}
J.-P. Kahane,
``Sur les r\'{e}arrangements de fonctions de la classe $A$'',
Studia Math. \textbf{31} (1968), 287--293. 

\bibitem{kahane:absolument}
J.-P. Kahane,
``S\'{e}ries de Fourier absolument convergentes'',
Springer-Verlag, 1970.

\bibitem{kahane:random}
J.-P. Kahane,
``Some random series of functions'', 
2nd ed., Cambridge University Press, 1985.

\bibitem{kahane:baire}
J.-P. Kahane,
``Baire's category theorem and trigonometric series'',
J. Anal. Math. \textbf{80} (2000), 143--182.

\bibitem{kahane-salem}
J.-P. Kahane and R. Salem,
``Ensembles parfaits et s\'eries trigonom\'etriques'',
2nd ed., Hermann, 1994. 

\bibitem{katznelson}
Y. Katznelson,
``An introduction to harmonic analysis'', 3rd ed.,
Cambridge University Press, 2004.


\bibitem{kaufman:functional}
R. Kaufman,
``A functional method for linear sets'',
Israel J. Math. \textbf{5} (1967), 185--187. 

\bibitem{kaufman:msets}
R. Kaufman,
``M-sets and distributions'',
Ast\'erisque \textbf{5}, Soc. Math. France (1973), 225--230.

\bibitem{kechris-louveau}
A. S. Kechris and A. Louveau,
``Descriptive set theory and the structure of sets of uniqueness'',
Cambridge University Press, 1987.

\bibitem{kinukawa}
M. Kinukawa,
``A note on the closure of translations in $L^p$'',
T\^{o}hoku Math. J. (2) \textbf{18} (1966), 225--231. 

\bibitem{korner:pseudofunc}
T. W. K\"{o}rner,
``A pseudofunctions on a Helson set'', I and II,
Ast\'erisque \textbf{5}, Soc. Math. France (1973), 3--224 and 231--239.

\bibitem{korner:ivashev-musatov}
T. W. K\"{o}rner,
``On the theorem of Iva\v{s}ev-Musatov III'',
Proc. London Math. Soc. (3) \textbf{53} (1986), 143--192. 

\bibitem{korner:curve}
T. W. K\"{o}rner,
``Kahane's Helson curve'',
J. Fourier Anal. Appl., Special Issue (1995), 325--346. 

\bibitem{krasno-rutickii}
M. A. Krasnosel'skii and Ya. B. Rutickii,
``Convex functions and Orlicz spaces'',
P. Noordhoff, Groningen, 1961.

\bibitem{lev-olevskii:piatetski}
N. Lev and A. Olevskii,
``Piatetski-Shapiro phenomenon in the uniqueness problem'',
C. R. Acad. Sci. Paris, Ser. I \textbf{340} (2005), 793-798.

\bibitem{lev-olevskii:generators}
N. Lev and A. Olevskii,
``No characterization of generators in $\ell^p$ $(1 < p < 2)$ by zero set of Fourier transform'',
C. R. Acad. Sci. Paris, Ser. I \textbf{346} (2008), 645-648.

\bibitem{linden-tzaf}
J. Lindenstrauss and L. Tzafiri,
``Classical Banach Spaces'', Volume I, Springer-Verlag, 1977.



\bibitem{newman}
D. J. Newman,
``The closure of translates in $l^p$'',
Amer. J. Math. \textbf{86} (1964), 651--667. 

\bibitem{olevskii:localization}
A. M. Olevskii,
``The localization of Carleman singularities on compact sets of measure zero'' (in Russian),
Dokl. Akad. Nauk SSSR \textbf{202} (1972), 30--33.
English translation in Soviet Math. Dokl. \textbf{13} (1972), 27--30.

\bibitem{olevskii:general}
A. M. Olevskii,
``Fourier series with respect to general orthogonal systems'',
Springer-Verlag, 1975 (English translation).

\bibitem{petrov}
V. V. Petrov,
``Sums of independent random variables'',
Springer, 1975.


\bibitem{piatetskii-1}
I. I. Pyatecki\u\i-\v Sapiro,
``On the problem of uniqueness of
expansion of a function in a trigonometric series'' (in Russian),
Moskov. Gos. Univ. U\v c. Zap. Mat. \textbf{155(5)} (1952), 54--72.
English translation in ``Selected works of Ilya Piatetski-Shapiro'',
AMS Collected Works 15, 2000.

\bibitem{piatetskii-2}
I. I. Pyatecki\u\i-\v Sapiro,
``Supplement to the work: On the problem of uniqueness of
expansion of a function in a trigonometric series'' (in Russian),
Moskov. Gos. Univ. U\v c. Zap. Mat. \textbf{165(7)} (1954), 79--97.
English translation in ``Selected works of Ilya Piatetski-Shapiro'',
AMS Collected Works 15, 2000.

\bibitem{pollard}
H. Pollard,
``The closure of translations in $L^p$'',
Proc. Amer. Math. Soc. \textbf{2} (1951), 100--104. 

\bibitem{rosenblatt-shuman}
J. M. Rosenblatt and K. L. Shuman,
``Cyclic functions in $L^p(\R)$, $1 \leq p <\infty$'',
J. Fourier Anal. Appl. \textbf{9} (2003), 289--300.

\bibitem{rudin}
W. Rudin, ``Real and complex analysis'',
3rd ed., McGraw-Hill, New York, 1987.

\bibitem{salem}
R. Salem,
``On singular monotonic functions whose spectrum has a given Hausdorff dimension'',
Ark. Mat. \textbf{1} (1951), 353--365. 

\bibitem{segal:span}
I. Segal,
``The span of the translations of a function in a Lebesgue space'',
Proc. Nat. Acad. Sci. U. S. A. \textbf{30} (1944), 165--169.




\bibitem{wiener}
N. Wiener,
``The Fourier integral and certain of its applications'',
Cambridge University Press, 1933. Reprint, Dover Publications, 1959.

\bibitem{zygmund}
A. Zygmund, ``Trigonometric series'', 2nd ed.,
Cambridge University Press, 1959.

\end{thebibliography}
\end{document}